\documentclass[12pt]{amsart}
\usepackage{amsmath,amsfonts,amssymb,amsthm}
\usepackage{graphicx,color}
\usepackage{hyperref}

\usepackage{cleveref}

\DeclareMathRadical{\sqrtsign}{symbols}{"70}{largesymbols}{"70}

\newlength{\figboxwidth}             
\newtheorem{theorem}{Theorem}[section]

\newtheorem{lemma}[theorem]{Lemma}

\theoremstyle{definition}

\theoremstyle{remark}

\theoremstyle{plain}

\theoremstyle{plain}
\newtheorem{corollary}{Corollary}

\numberwithin{equation}{section}

\newcommand{\abs}[1]{\left\lvert#1\right\rvert}

\def\Diff{\operatorname{Diff}}

\newcommand{\N}{\mathbb{N}}
\newcommand{\Z}{\mathbb{Z}}
\newcommand{\R}{\mathbb{R}}
\newcommand{\Q}{\mathbb{Q}}

\begin{document}
	
	\title[Second Lagrange spectra]
	{On the geometry of the second Lagrange spectra}

	\author{Hao Cheng}
	\address{IMPA, Estrada Dona Castorina 110, 22460-320, Rio de Janeiro, Brazil}
	\email{hao.cheng@impa.br}
	\thanks{The first author was partially supported by CNPq.}

	\author{Harold Erazo}
	\address{IMPA, Estrada Dona Castorina 110, 22460-320, Rio de Janeiro, Brazil}
	\email{harold.eraz@gmail.com}
	\thanks{The second author was partially supported by CAPES and FAPERJ}

	\author{Carlos Gustavo Moreira}
	\address{Carlos Gustavo Moreira: SUSTech International Center for Mathematics, Shenzhen, Guangdong, People’s Republic of China and IMPA, Estrada Dona Castorina 110, 22460-320, Rio de Janeiro, Brazil
	}
	\email{gugu@impa.br}
	\thanks{The third author was partially supported by CNPq and FAPERJ}
	
	\author{Thiago  Vasconcelos}
	\address{IMPA, Estrada Dona Castorina 110, 22460-320, Rio de Janeiro, Brazil}
	\email{thiago.greco@impa.br}
	\thanks{The fourth author was partially supported by CNPq}


	\date{\today}

	
	\begin{abstract}
		The Lagrange spectrum $L$ is the set of finite values of the best approximation constants $k(\alpha)=\limsup_{|p|,|q|\to \infty}|q(q\alpha-p)|^{-1}$, where $\alpha\in \mathbb{R}\setminus \mathbb{Q}$. It is a classical result that the pairs $(p,q)$ attaining these approximation constants arise from the convergents $(p_n,q_n)$ of the continued fraction of $\alpha$. Consequently, $k(\alpha)=\limsup_{n\to\infty}|q_n(q_n\alpha-p_n)|^{-1}$. Moreira proved that the function $d(t)=HD(L\cap(-\infty,t))$ where $HD$ denotes Hausdorff dimension, is continuous. 
		
		Second Lagrange spectra are defined analogously to the classical Lagrange spectrum, but are associated with the problem of approximating an irrational number $\alpha$ by rational numbers $\frac{p}{q}$ that are not convergents of its continued fraction expansion. Two natural definitions arise depending on whether rational multiples $(p,q)=(kp_n,kq_n),k\geq 2$ which represent the same rational numbers as convergents, are allowed or excluded. Based on this distinction, Moshchevitin introduced two second Lagrange spectra, denoted $L_2$ and $L_2^*$. 
		
		We prove that the function $d_2(t)=HD(L_2\cap (-\infty,t))$ is continuous, whereas $d_2^*(t)=HD(L_2^*\cap (-\infty,t))$ is discontinuous and assumes only the values 0 and 1.
	\end{abstract}

	\maketitle
	
	\tableofcontents
	
	\section{Introduction and statement of the results}
	
	\subsection{Classical spectra}
	
	The classical Lagrange and Markov spectra are closed subsets of the real line related to Diophantine approximations. They arise naturally in the study of rational approximations of irrational numbers and of indefinite binary quadratic forms, respectively. More precisely, given an irrational number $\alpha$, let 
	$$k(\alpha):=\limsup_{\substack{|p|, |q|\to\infty \\ p, q\in\mathbb{Z}}}(q|q\alpha-p|)^{-1}$$
	be its best constant of Diophantine approximation. The set 
	$$L:=\{k(\alpha):\alpha\in\mathbb{R}-\mathbb{Q}, k(\alpha)<\infty\}$$ 
	consisting of all finite best constants of Diophantine approximations is the so-called \emph{Lagrange spectrum}. 
	
	The study of the Lagrange spectrum is closely related to the theory of continued fractions. This is because of the classical fact that the best approximations of an irrational number $\alpha$ must come from the convergents of its continued fraction. In other words, if $\frac{p_n}{q_n}$ are the convergents of the continued fraction of $\alpha$, then $$k(\alpha)=\limsup_{n\to \infty}(q_n|q_n\alpha-p_n|)^{-1}$$ Moreover, if $\alpha=[a_0;a_1,a_2\cdots]=a_0+\frac{1}{a_1+\frac{1}{a_2+\frac{1}{\ddots}}}$, is the continued fraction expansion of $\alpha$, then we have the more explicit formula $$k(\alpha)=\limsup_{n\to \infty}(\alpha_{n}+\beta_{n})$$
	
	where $\alpha_n=[a_n;a_{n+1},\cdots]$ and $\beta_n=[0;a_{n-1},a_{n-2}\cdots,a_1]$.

	Similarly, the \emph{Markov spectrum} 
	$$M:=\left\{\left(\inf\limits_{(x,y)\in\mathbb{Z}^2\setminus\{(0,0)\}} |q(x,y)|\right)^{-1} < \infty: q(x,y)=ax^2+bxy+cy^2, b^2-4ac=1\right\} $$
	consists of the reciprocal of the minimal values over non-trivial integer vectors $(x,y)\in\mathbb{Z}^2\setminus\{(0,0)\}$ of indefinite binary quadratic forms $q(x,y)$ with discriminant one. 
	
	Hurwitz showed that the minimum of $L$ is $\sqrt{5}.$ Markov \cite{Markoff1879, Markoff1880} improved the result by showing that 
	$$L\cap (-\infty, 3)=M\cap (-\infty, 3)=\left\{k_1=\sqrt{5}<k_2=2\sqrt{2}<k_3=\frac{\sqrt{221}}{5}<\dots\right\},$$
	where $k^2_n\in \mathbb{Q}$ for every $n\in \mathbb{N}$ and $k_n\to 3$ when $n\to \infty$. Hall \cite{Hall} showed that $L$ contains a half-line $[c,+\infty)$ for some $c>4$, and Freiman \cite{Freiman} determined the biggest such a half-line. We refer the reader to \cite{CF89, CDMLS} for more information regarding these classical spectra.
	
	Moreira in \cite{M3} proved several results on the geometry of the Markov and Lagrange spectra, for example, that the map $d:\mathbb{R} \rightarrow [0,1]$, given by
	$$
	d(t)=HD(L\cap(-\infty,t))= HD(M\cap(-\infty,t))
	$$
	is continuous, surjective and such that $d(3)=0$ and $d(\sqrt{12})=1$. Moreover, that
	$$d(t)=\min \{1,2D(t)\}$$
	where $D(t)=HD(k^{-1}(-\infty,t))=HD(k^{-1}(-\infty,t])$ is also a continuous function from $\mathbb{R}$ to $[0,1).$ Even more, he showed the limit
	$$\lim_{t\rightarrow \infty}HD(k^{-1}(t))=1.$$
	
	A crucial tool used in the proof of this theorem is the following formula for the Hausdorff dimension of the arithmetic sum of regular Cantor sets which was proved by Moreira in \cite{M2}:
	
	\begin{theorem}[Moreira]
		Let $K,K'$ be $C^2$ regular Cantor sets and suppose that $K$ is non essentially affine. Then, $HD(K+K')=\min\{1,HD(K)+HD(K')\}$.
	\end{theorem}
	In this work, we will investigate the behaviour of the analogue of the dimension function $d(t)$ for variations of the Lagrange spectrum related to approximations of an irrational number $\alpha$ by rational numbers $\frac{p}{q}$ which are not convergents of the continued fraction of $\alpha$. They will be defined later.
	
	\subsection{Dynamical spectra}
	
	The Lagrange and Markov spectra have a dynamical characterizaton, which was first obtained by Perron.
	
	Denote by $[a_0,a_1,\dots]$ the continued fraction $a_0+\frac{1}{a_1+\frac{1}{\ddots}}$. Let $\Sigma=(\mathbb{N}_{>0})^{\mathbb{Z}}$ the space of bi-infinite sequences of positive integers, $\sigma:\Sigma\to\Sigma$ be the left-shift map $\sigma((a_n)_{n\in\mathbb{Z}}) = (a_{n+1})_{n\in\mathbb{Z}}$, and let $f:\Sigma\to\mathbb{R}$ be the function
	$$f((a_n)_{n\in\mathbb{Z}}) = \alpha_0+\beta_0^*=[a_0, a_1,\dots] + [0, a_{-1}, a_{-2},\dots].$$
	Then, 
	$$L=\left\{\limsup_{n\to\infty}f(\sigma^n(\underline{\theta})):\underline{\theta}\in\Sigma\right\} \quad \textrm{and} \quad M= \left\{\sup_{n\to\infty}f(\sigma^n(\underline{\theta})):\underline{\theta}\in\Sigma\right\}.$$
	It follows from these characterizations that $M$ and $L$ are closed subsets of $\mathbb R$ and that $L\subset M$.  
	
	This dynamical description of the spectra motivate the general definitions of dynamical Markov and Lagrange spectra. Let $\varphi:X\to X$ be a continuous map defined on a metric space $X$ and $f:X\to \mathbb{R}$ be a continuous function. Following the above characterization of the classical spectra, we define the functions
	\begin{eqnarray*}
		\ell_{\varphi,f}: X &\rightarrow& \mathbb{R} \\
		x &\mapsto& \ell_{\varphi,f}(x)=\limsup_{n\to \infty}f(\varphi^n(x)),
	\end{eqnarray*}
	\begin{eqnarray*}
		m_{\varphi,f}:X &\rightarrow& \mathbb{R} \\
		x &\mapsto& m_{\varphi,f}(x)=\sup_{n\to \infty}f(\varphi^n(x)),
	\end{eqnarray*} and the sets
	
	$$L_{\varphi,f,X}=\{\ell_{\varphi,f}(x):x\in X\}$$ and $$M_{\varphi,f,X}=\{m_{\varphi,f}(x):x\in X\}$$
	If $HD(X)$ denotes the Hausdorff dimension of $X$, we define
	\begin{equation}\label{f1}
		L(t)=L(\varphi,f,X)(t)=HD(L_{\varphi,f,X}\cap (-\infty,t))
	\end{equation}
	and
	\begin{equation}
		M(t)=M(\varphi,f,X)(t)=HD(M_{\varphi,f,X}\cap (-\infty,t)).
	\end{equation}
	
	The dynamical characterization of the classical Markov and Lagrange spectra can be also be done in a smooth setting via an invertible extension of the Gauss map first considered by Adler and Flatto \cite{AdlerFlatto}. 
	
	Define the map $$T(x,y)=\left(\frac{1}{x}-\left\lfloor \frac{1}{x}\right\rfloor, \frac{1}{\left\lfloor \frac{1}{x}\right\rfloor +y }\right)$$ in $(\mathbb{R}\setminus \mathbb{Q})^2$ and considering the height function $f(x,y)=\frac{1}{x}+y$, then the dynamical Markov and Lagrange spectra of $f$ associated to $T$ coincide with the classical spectra. Nakada, Ito and Tanaka showed in \cite{NIT} map $T$ admits the invariant density $d\mu=\frac{1}{(1+xy)^2}dxdy$.
	
	Defining the Cantor sets $C(k)=\{x=[0;a_1,a_2,\cdots]:a_n\leq k\}$ for $k\in \mathbb{N}_{>0}$, $C(k)\times C(k)$ is a conservative horseshoe for $T$. Moreira and Lima \cite{ML} proved that the dynamical Markov and Lagrange of $T$ with respect to this horseshoe coincide with $M\cap (-\infty,\sqrt{k^2+4k}]$ and $L\cap (-\infty,\sqrt{k^2+4k}]$, respectively, for $k\geq 4$. 
	
	Because of the existence of Hall's ray, this means that the nontrivial parts of the Markov and Lagrange spectra coincide with dynamical Markov and Lagrange spectra associated to conservative horseshoes. Having this in mind, Cerqueira, Matheus, Moreira \cite{CerqueiraMatheusGugu} and Lima, Moreira, Villamil \cite{LMV} proved an analogue of Moreira's theorem about the continuity of the dimension function $d(t)$ for dynamical Markov and Lagrange spectra associated to two-dimensional conservative horseshoes.
	
	Let $\Diff^{2}_{\omega}(S)$ denote the set of conservative diffeomorphisms of a surface $S$ with respect to a volume form $\omega$. 
	They study the Hausdorff dimension of the sets $L_{\varphi,f,\Lambda}\cap (-\infty,t)$ and $M_{\varphi,f,\Lambda}\cap (-\infty,t)$ by studying the dimension of the subsets, where $\varphi$ is a diffeomorphism on $S$, $f$ is a differentiable real function on $S$ and $\Lambda$ is a horseshoe associated to the diffeomorphism. Define the one parameter family of maximal invariant hyperbolic sets $$\Lambda_t:=\bigcap\limits_{n\in\mathbb{Z}}\varphi^{-n}(\{y\in\Lambda: f(y)\leq t\}) = \{x\in\Lambda: m_{\varphi, f}(x)=\sup\limits_{n\in\mathbb{Z}}f(\varphi^n(x))\leq t\}.$$ Now define the stable and unstable Cantor sets via the projections along unstable and stable manifolds respectively:
	$$K^u_t=\bigcup_{a\in \mathcal{A}} \pi^s_a(\Lambda_t\cap R_a) \ \mbox{and} \ K^s_t=\bigcup_{a\in \mathcal{A}}\pi^u_a(\Lambda_t\cap R_a).$$ 
	Their result is the following:
	
	\begin{theorem}[Lima-Moreira-Villamil]\label{main}
		Let $\varphi_0\in \Diff^{2}_{\omega}(S)$ with a horseshoe $\Lambda_0$ and $\mathcal{U}$ a $C^{2}$-sufficiently small neighbourhood of $\varphi_0$ in $\Diff^{2}_{\omega}(S)$ such that $\Lambda_0$ admits a continuation $\Lambda (= \Lambda(\varphi))$ for every $\varphi \in \mathcal{U}$. There exists a residual set $\tilde{\mathcal{U}}\subset \mathcal{U}$ such that for every $\varphi\in \tilde{\mathcal{U}}$ and $r\ge 2$ there exists a $C^r$-residual set $\tilde{\mathcal{R}}_{\varphi,\Lambda}\subset C^r(S,\mathbb{R})$ such that for $f\in \tilde{\mathcal{R}}_{\varphi,\Lambda}$  the functions:
		$$t\mapsto d_u(t):= HD(K^u_t) \ \mbox{and} \ t\mapsto d_s(t):=HD(K^s_t)$$
		are continuous and in fact, they are equal with
		$$HD(\Lambda_t)=d_u(t)+d_s(t)=2d_u(t)$$ and 
		$$\min\{1,HD(\Lambda_t)\}=L(t)=M(t).$$
	\end{theorem}
	
	It is not hard to verify that the generic conditions of the theorem hold for the case of horseshoes of the form $C(k)\times C(k)$ for the map $T$ and function $f$. Indeed, $\Lambda$ will check the generic conditions in $\tilde{\mathcal{U}}$ because the map $T$ is the natural extension of the Gauss map and that Gauss-Cantor sets are non-essentially affine \cite[Proposition 1]{M3}, and that condition is a generalization of this property. The function $f$ will satisfy the necessary conditions because its gradient is not perpendicular to the stable or to the unstable direction. These facts will be used later in this paper to verify the validity of the generic hypotheses for the second Lagrange spectra (see \Cref{cor:continuity}).
	
	\subsection{Second Lagrange spectra}
	
	Moshchevitin considered in \cite{Mos} the problem of approximating an irrational number $\alpha$ by rational numbers that are not convergents of the continued fraction of $\alpha$. More precisely, if $(\frac{p_n}{q_n})_{n\geq 0}$ are the convergents of the continued fraction of $\alpha$, he defined the second constant of approximation 
	$$k_2(\alpha):=\underset{(p,q)\neq(p_{n},q_{n}),\forall n\in \mathbb{N}}{\underset{|p|,|q|\rightarrow \infty}{\limsup}}|q(q\alpha-p)|^{-1}$$
	and the second Lagrange spectrum $$L_2:=\{k_2(\alpha):\alpha\in \mathbb{R}\setminus \mathbb{Q},k_2(\alpha)<\infty\}.$$ He proved that the smallest element of $L_2$ is $\lambda_1=\frac{\sqrt{5}}{4}=k_2(\frac{1+\sqrt{5}}{2})$ and the second smallest element is $\lambda_2=\frac{\sqrt{17}}{4}=k_2(\frac{1+\sqrt{17}}{2})$. Also, if $k_2(\alpha)=\lambda_1$ or $k_2(\alpha)=\lambda_2$, then $\alpha$ is equivalent to $\frac{1+\sqrt{5}}{2}=[1,\overline{1}]$ or $\frac{1+\sqrt{17}}{4}=[2;\overline{1,1,3}]$, respectively.  To prove this, Moshchevitin proved that 
	
	\begin{lemma}[Moshchevitin]
		For $\alpha\in\R\setminus\Q$, we have $k_2(\alpha)=\limsup_{k\to \infty}\delta(k)$, where 
		
		\begin{equation*}
			\delta(k):=
			\begin{cases}
				\max\left\{\frac{\alpha_{k+1}+\beta_{k+1}}{4},\frac{\alpha_{k+1}+\beta_{k+1}}{(\alpha_{k+1}-1)(1+\beta_{k+1})},\frac{\alpha_{k+1}+\beta_{k+1}}{(\alpha_{k+1}-a_{k+1}+1)(a_{k+1}-1+\beta_{k+1})}\right\}, & \text{if } a_{k+1}\geq2\\
				\frac{\alpha_{k+1}+\beta_{k+1}}{4}, &\text{if } a_{k+1}=1
			\end{cases}
		\end{equation*}
	\end{lemma}
	
	We will give a different proof of Moshchevitin formula in \Cref{chap:second_spectra_formulas}. 
	
	Later, P. Semenyuk \cite{Sem} proved that the third smallest element is $\lambda_3=\frac{13\sqrt{173}}{164}$ and that if $k_2(\alpha)=\lambda_3$, then $\alpha$ is equivalent to $\frac{39+13\sqrt{173}}{82}=[2;\overline{1,1,1,1,3,1,1,3}]$.
	
	Gayfulin \cite{Gay} proved that, similarly to the classical Lagrange spectrum, the beginning of $L_2$ is also discrete. He gave the following description of the discrete part of this spectrum:
	
	\begin{theorem}[Gayfulin]
		
		For $n\geq 3$, define $\xi_n=[0;\overline{1,1,1,1,3,(1,1,3)_{2n-5}}]$ and \\ $$\lambda_n=\frac{[3;\overline{(1,1,3)_{2n-5},1,1,1,1,3}]+[0;\overline{1,1,1,1,(3,1,1)_{2n-5}}]}{4}.$$ Let $$\lambda_{\infty}=\frac{3\sqrt{17}+21}{32}=\lim_{n\to \infty}\lambda_n=\frac{3\sqrt{17}+21}{32}.$$ 
		
		Then, $\lambda_{\infty}$ is the first accumulation point of $L_2$ and $L_2\cap (-\infty,\lambda_{\infty})=\{\lambda_1<\lambda_2<\lambda_3<\cdots \}$. Also, if $k_2(\alpha)=\lambda_n$, then $\alpha$ is equivalent to $\xi_n$.
	\end{theorem}
	This spectrum can also be represented dynamically. Consider again the shift $\sigma:\Sigma\to \Sigma$ defined over the space $\Sigma=(\mathbb{N}_{>0})^{\mathbb{Z}}$ and the function $f:\Sigma\to \mathbb{R}$ defined by 
	$$f((a_n)_{n\in \mathbb{Z}}):=
	\begin{cases}
		\max\left\{\frac{\alpha_{0}+\beta_{0}^*}{4},\frac{\alpha_{0}+\beta_{0}^*}{(\alpha_{0}-1)(1+\beta_{0}^*)},\frac{\alpha_{0}+\beta_{0}^*}{(\alpha_{0}-a_{0}+1)(a_{0}-1+\beta_{0}^*)}\right\}, &\text{if }a_{0}\geq2\\
		\frac{\alpha_{0}+\beta_{0}^*}{4}, &\text{if } a_{0}=1,
	\end{cases}
	$$
	
	where $\alpha_0=[a_0,a_1,a_2,\cdots]$ and $\beta_0^*=[0,a_{-1},a_{-2},\cdots]$.

	In analogy to the Lagrange spectrum, the nontrivial parts of this spectrum will also be reduced to dynamical Lagrange spectra associated to horseshoes of the map $T$ defined earlier. In this paper we will prove that the dimension function $d_2(t)=HD((L_2\cap (-\infty,t))$ is continuous.
	
	The main difference between this second Lagrange spectrum and the classical Lagrange spectrum is that in this case the height function is no longer a smooth function, but rather the maximum of a finite set of smooth functions.

	Therefore, in order to prove that $d_2$ is continuous, we will prove a that the theorem of continuity of dimension for dynamical Markov and Lagrange spectra associated to conservative horseshoes also holds when the function $f$ may no longer be a smooth function but rather the maximum of a finite number of smooth functions. This version can be stated as follows:

	\begin{theorem}\label{thm:main}
		
		Let $\varphi_0\in \Diff^{2}_{\omega}(S)$ with a horseshoe $\Lambda_0$ and $\mathcal{U}$ a $C^{2}$-sufficiently small neighbourhood of $\varphi_0$ in $\Diff^{2}_{\omega}(S)$ such that $\Lambda_0$ admits a continuation $\Lambda (= \Lambda(\varphi))$ for every $\varphi \in \mathcal{U}$. There exists a residual set $\tilde{\mathcal{U}}\subset \mathcal{U}$ such that for every $\varphi\in \tilde{\mathcal{U}}$ and $r\ge 2$ there exists a $C^r$-residual set $\tilde{\mathcal{R}}_{\varphi,\Lambda}\subset C^r(S,\mathbb{R})$ such that for $f=\max_{i\leq N}f_i, f_i\in \tilde{\mathcal{R}}_{\varphi,\Lambda}$  the functions:
		$$t\mapsto d_u(t):= HD(K^u_t) \ \mbox{and} \ t\mapsto d_s(t):=HD(K^s_t)$$
		are continuous and in fact, they are equal with
		$$HD(\Lambda_t)=d_u(t)+d_s(t)=2d_u(t)$$ and 
		$$\min\{1,HD(\Lambda_t)\}=L(t)=M(t).$$
	\end{theorem}
	
	\begin{corollary}\label{corE}
		Defining $d_2(t):=HD(L_2\cap (-\infty,t))$, then $d_2(t)=\overline{\dim}_{box}(L_2\cap (-\infty,t))$, and this function is continuous on $t\in\mathbb{R}$ and surjective on $[0,1]$.
	\end{corollary}
	
	The definition of the approximation constant used to $L_2$ allows the possibility that the second best approximations are given by fractions $\frac{p}{q}$ that are equal to a convergent $\frac{p_n}{q_n}$, with $(p,q)=k(p_n,q_n)$. This lead Moshchevitin to formulate an alternative definition of the second Lagrange spectrum. We define $k_2^*(\alpha):=\limsup_{\frac{p}{q}\neq \frac{p_n}{q_n}}(|q||q\alpha-p|)^{-1}$, $L_2^*:=\{k_2^*(\alpha):\alpha\in \mathbb{R}\setminus \mathbb{Q},k_2^*(\alpha)<+\infty\}$.
	
	He proved that $L_2^*\cap (-\infty,\frac{2}{3})=\{\frac{\sqrt{5}}{5}\}=\{k_2^*(\frac{-1+\sqrt{5}}{2})\}$, and that the first accumulation point is $\frac{2}{3}=k_2^*(e)$ where $e$ is the Euler number. Contrary to the classical Lagrange spectrum $L$ and the second Lagrange spectrum $L_2$, the set $L_2^{*}$ is bounded $L_2^{*}\subset[0,2]$ (because of a classical theorem of Legendre) and instead of containing a Hall's ray, it ends with the interval $[\frac{3}{2},2]\subset L_2^{*}$. The structure of $L_2^{*}\cap(\frac{2}{3},\frac{3}{2})$ is unknown besides from the fact that it is infinite.
	
	Moshchevitin also found a formula for $k_2^*$. He proved that $k_2^*(\alpha)=\limsup_{n\to \infty}\kappa(\alpha)$, where $$\kappa(k):=\begin{cases}
		\frac{\alpha_{k+1}+\beta_{k+1}}{(2\alpha_{k+1}-1)(2\beta_{k+1}+1)}, & \text{if } a_{k+1}=1 \\
		\max\left\{\frac{\alpha_{k+1}+\beta_{k+1}}{(\alpha_{k+1}-1)(1+\beta_{k+1})},\frac{\alpha_{k+1}+\beta_{k+1}}{(\alpha_{k+1}-a_{k+1}+1)(a_{k+1}-1+\beta_{k+1})}\right\}, & \text{if } a_{k+1}\geq 2 
	\end{cases}$$
	
	We have a variation of the formula, again, with a different proof:
	
	\begin{lemma}
		For $\alpha\in \mathbb{R}\setminus \mathbb{Q}$, we have $k_2^*(\alpha)=\limsup_{n\to \infty}\gamma(n)$, where
		$$\gamma(k):=
		\begin{cases}
			\max\left\{\frac{\alpha_{k+1}+\beta_{k+1}}{(2-\beta_{k+1})(\alpha_{k+1}+2)},\frac{\alpha_{k+1}+\beta_{k+1}}{(\alpha_{k+1}-1)(1+\beta_{k+1})},\frac{\alpha_{k+1}+\beta_{k+1}}{(\alpha_{k+1}-a_{k+1}+1)(a_{k+1}-1+\beta_{k+1})}\right\}, &\text{if }a_{k+1}\geq2\\
			\frac{\alpha_{k+1}+\beta_{k+1}}{(2-\beta_{k+1})(\alpha_{k+1}+2)}, & \text{if }a_{k+1}=1.
		\end{cases}
		$$
	\end{lemma}

	We will use both formulas, according to convenience.
	With this formula, the spectrum $L_2^*$ also admits a dynamical characterization as the dynamical Lagrange spectrum for the function $f^*:\Sigma\to \mathbb{R}$, $$f^*((a_n)_{n\in \mathbb{Z}})=\begin{cases}
		\frac{\alpha_{0}+\beta_{0}^*}{(2\alpha_{0}-1)(2\beta_{0}^*+1)}, & \text{if } a_{0}=1 \\
		\max\left\{\frac{\alpha_{0}+\beta_{0}^*}{(\alpha_{0}-1)(1+\beta_{{0}}^*)},\frac{\alpha_{0}+\beta_{0}^*}{(\alpha_{0}-a_{0}+1)(a_{0}-1+\beta_{0}^*)}\right\}, & \text{if } a_{0}\geq 2 
	\end{cases}$$
	for the shift $\sigma:\Sigma\to \Sigma$. 
	
	However, contrary to $L_2$, the dimension function $d_2^*(t)=HD(L_2^*(-\infty,t))$ is not continuous. In fact, we have the following result:
	
	\begin{theorem}\label{thm:discontinuous}
		$d_2^*(\frac{2}{3})=0$, but for every $\epsilon>0$, $d_2^*(\frac{2}{3}+\varepsilon)=1$.
	\end{theorem}
	
	The first statement of the theorem follows immediately from the fact that $L_2^*\cap (-\infty,\frac{2}{3})=\{\frac{\sqrt{5}}{5}\}$ is discrete. To prove the second statement, we use the remark by Moshchevitin that the particular expression of $e=[2;1,2,1,1,4,1,1,6,1,1,\cdots]$ is not crucial to obtain $k_2^*(e)=\frac{2}{3}$. Any number $\alpha$ equivalent to a number of the form $[0;1,1,x_3,1,1,x_6,\cdots]$ with $\lim_{n\to \infty}x_{3n}=\infty$ will be such that $k_2^*(\alpha)=\frac{2}{3}$. We obtain a converse to this result, which provides a clearer understanding of $L_2^{*}\cap(\frac{2}{3},\frac{2}{3}+\varepsilon)$. 
	
	\begin{lemma}\label{lem:reciprocal_lemma}
		If $\alpha\in\R\setminus\Q$ is such that $k_2^{*}(\alpha)\in[\frac{2}{3},\frac{\sqrt{493}}{33})$. Then $\alpha$ is equivalent to $[0;1,1,x_{3},1,1,x_6,\cdots]$ where $x_{3n}\geq 145$ for all $n\geq 1$. Moreover, these constants are optimal because 
		\begin{equation*}
			k_2^{*}([0;\overline{1,1,1,1,21}])=\frac{\sqrt{493}}{33}=0.67283646\dots
		\end{equation*}
		\begin{equation*}
			k_2^*([0;\overline{1,1,145}])=\frac{\sqrt{21317}}{217}=0.67282684\dots<\frac{\sqrt{493}}{33}.
		\end{equation*}
	\end{lemma}

	Motivated by these facts, we consider the Gauss-Cantor sets $$X(k)=\{x=[0;1,1,x_3,1,1,x_6,\cdots,]|k^2\leq x_{3n}\leq k^3\}$$ for $k\in \mathbb{N}_{>0}$.

	For every $\varepsilon>0$, if $k$ is large enough, we will have that $k_2^*(x)\leq \frac{2}{3}+\varepsilon$. The set $\Lambda(k)=G^2(X(k))\times X(k)$ is a horseshoe for the map $T^3$, where $G$ denotes the Gauss map. In analogy to the expression for $f$ given above, we will define a smooth by parts function which we also call $\tilde{f}:\Lambda(k)\to \mathbb{R}$ such that $L_{T^3,\tilde{f},\Lambda(k)}\subset L_2^*\cap (-\infty,\frac{2}{3}+\varepsilon)$. We will also prove that for large $k$, $HD(X(k))>\frac{1}{2}$, and therefore $HD(\Lambda(k))>1$. By the third statement in \Cref{thm:main}, we will have that $HD(L_2^*\cap (-\infty,\frac{2}{3}+\varepsilon))=1$.
	
	The main conceptual reason that the continuity of dimension fails for this case is that the analysis of this spectrum cannot be reduced to studying dynamical Lagrange spectra associated to horseshoes. This will happen because the numbers in Cantor sets $X(k)$ have very large digits in the continued fraction expansion, so the interesting parts of this spectrum cannot be reduced to dynamical Lagrange spectra associated to a horseshoe, which is associated to a finite type subshift of $\Sigma$.
	
	\subsection{Structure of the paper}
	
	The paper is organized as follows.
	In Section 2, we extend the result of Lima--Moreira--Villamil to the setting in which the function is given by the maximum of finitely many smooth functions.
	In Section 3.1, we provide a new proof of the formulas for $k_2(\alpha)$ and $k_2^\ast(\alpha)$.
	In Section 3.2, we describe the dynamical characterization of $L_2$ and $L_2^\ast$ as dynamical Lagrange spectra.
	In Section 3.3, we prove the continuity of the dimension function for $L_2$.
	In Section 3.4, we establish the discontinuity stated in \Cref{thm:discontinuous}.
	In Section 3.5, we prove \Cref{lem:reciprocal_lemma}.
	Finally, in Section 3.6, we prove the lower bound $HD(X(k))>\frac{1}{2}$ for the Hausdorff dimension of $X(k)$ for large $k$, which is used in Section 3.4.

	\section{Proof of the continuity of dimension for maxima of smooth functions}
	
	In this section we explain how to extend the result of Lima, Moreira and Villamil for the case where the height function $f$ is of the form $f=\max_{1\leq j\leq N} f_j$, where each $f_j$ satisfies adequate generic condition for that theorem. In their paper, the residual $\tilde{\mathcal{R}}_{\varphi,\Lambda}$ is the set $\tilde{\mathcal{R}}_{\varphi,\Lambda}=\bigcap_{\tilde{\Lambda}\subset \Lambda \text{ subhorseshoe}}H_{\tilde{\Lambda}}\cap \mathcal{R}_{\varphi,\Lambda}$, where $\mathcal{R}_{\varphi,\Lambda}=\{f\in C^r(S,\mathbb{R}):\nabla f(z)\neq 0 \forall z\in \Lambda \}$ and $H_{\tilde{\Lambda}}=\{f\in C^r(S,\mathbb{R}):|M_{f,\tilde{\Lambda}}|=1 \text{ and } Df_z(e_z^{s,u})\neq 0,\forall z\in M_{f,\tilde{\Lambda}}\}$, where $M_{f,\tilde{\Lambda}}$ is the set of maxima of $f$ in $\Lambda$. However, in the definition of $H_{\tilde{\Lambda}}$, the condition that $|M_{f,\tilde{\Lambda}}|=1$ is not necessary. They use it to prove that, for a subhorseshoe $\tilde{\Lambda}$ and $f\in \tilde{\mathcal{R}}_{\varphi,\Lambda}$, $HD(\ell_{\varphi,f}(\tilde{\Lambda}))=HD(m_{\varphi,f}(\tilde{\Lambda}))=\min\{1,HD(\tilde{\Lambda})\}$. To prove this, they use a construction from Moreira-Romana \cite{MR}, where they consider the unique point $x_M\in \tilde{\Lambda}$ and construct a subhorseshoe $\tilde{\Lambda}^{\varepsilon}\subset \tilde{\Lambda}\setminus \{x_M\}$ with $HD(\tilde{\Lambda}^{\varepsilon})>HD(\tilde{\Lambda}) (1-\varepsilon)$. The maximum of $f|_{\tilde{\Lambda}}$ will be $f(x_M)-\delta=\max_{z\in \tilde{\Lambda}} f(z)-\delta$, for some $\delta>0$, so they can construct points in $\tilde{\Lambda}$ whose iterates become close to $x_M$ in a controlled set of positions, and the value of $f$ elsewhere is at most $f(x_M)-\delta/2$, so the Markov and Lagrange value are realized only by these iterates. This construction can also be done if we use instead $G_{\tilde{\Lambda}}=\{f\in C^r(S,\mathbb{R})|\text{ and }Df_z(e^{s,u}_z)\neq 0,\forall z\in M_{f,\tilde{\Lambda}}\}$ in the place of $H_{\tilde{\Lambda}}$. It will instead use proposition 1 from Lima-Moreira-Villamil. In the conservative case it says that, for $f\in \mathcal{R}_{\varphi,\Lambda}$, $t \in \mathbb{R}$, $\epsilon>0$, there exists $\delta>0$ and a subhorseshoe $\Lambda'\subset \Lambda_{t-\delta}$ such that $HD(\Lambda')>HD(\Lambda)(1-\varepsilon)$. Using this proposition for $\tilde{\Lambda}$ instead of $\Lambda$ and $t=\max_{z\in \tilde{\Lambda}}f(z)$, using $\tilde{\Lambda}'$ in the role of $\tilde{\Lambda}^{\varepsilon}$ and choosing any point $x_M\in M_{f,\tilde{\Lambda}}$, the same construction can be done. Therefore, we are going to prove it in the case where the horseshoe $\Lambda$ satisfies the same hypotheses as in Lima-Moreira-Villamil and $f_j\in \mathcal{R}_{\varphi,f}'$ for every $1\leq j\leq N$, where $\mathcal{R}_{\varphi,f}'=\bigcap_{\tilde{\Lambda}\subset \Lambda \text{ subhorseshoe}}G_{\tilde{\Lambda}}\cap \mathcal{R}_{\varphi,\Lambda}$
	
	Define $\hat{S}=S\times \{1,\cdots,N\}$, $\hat{\Lambda}=\Lambda\times \{1,\cdots,N\}$ and for $\varphi:S\to S$, define $\hat{\varphi}:\hat{S}\to \hat{S}$ by \[\hat{\varphi}(x,j)=\begin{cases}
		(x,j+1), \text{ if} & 1\leq j\leq N-1\\
		(\varphi(x),1), \text{ if} & j = N
	\end{cases}\]
	
	This is a Rokhlin tower extension of $\Lambda$. $\hat{\Lambda}$ is also a conservative horseshoe for $\hat{\varphi}$ on the surface $\hat{S}=S\times \{1,\cdots,N\}$ with the same hyperbolic splitting as $\Lambda$.
	
	We then define the function $\hat{f}:\hat{\Lambda}\to \mathbb{R}$ by $\hat{f}(x,j)=f_j(x)$. These data satisfy the corresponding generic hypotheses  $\hat{\Lambda}\in \tilde{\mathcal{U}}$ and $\hat{f}\in \mathcal{R}_{\hat{\varphi},\hat{\Lambda}}'$ for $\hat{\Lambda} \subset \hat{S}$.
	
	Now, we see that for $(x,j)\in \hat{\Lambda}$, we have that $$\limsup
	_{n\to +\infty}\hat{f}(\hat{\varphi}^n(x,j))=\limsup
	_{k\to +\infty}\max_{n\in [kN,(k+1)N-1]}\hat{f}(\hat{\varphi}^n(x,j))=\limsup
	_{k\to +\infty}\max_{1\leq l\leq N}f_l(\varphi^k(x))=\limsup
	_{n\to +\infty}f(\varphi^n(x))$$ and $$\sup
	_{n\in \mathbb{Z}}\hat{f}(\hat{\varphi}^n(x,j))=\sup
	_{k\in \mathbb{Z}}\max_{n\in [kN,(k+1)N-1]}\hat{f}(\hat{\varphi}^n(x,j))=\sup
	_{k\in \mathbb{Z}}\max_{1\leq l\leq N}f_l(\varphi^k(x))=\sup
	_{n\in \mathbb{Z}}f(\varphi^n(x))$$
	
	This implies that $L_{\varphi,f,\Lambda}=L_{\hat{\varphi},\hat{f},\hat{\Lambda}}$ and $M_{\varphi,f,\Lambda}=M_{\hat{\varphi},\hat{f},\hat{\Lambda}}$ and therefore the continuity of the dimension function of the spectra is reduced to the smooth case.

	This argument shows that the theorem of continuity of dimension holds for the generic case where the function is the maximum of a finite number of smooth functions. However, for this paper we are more interested in verifying these conditions for practical applications. 
	
	\begin{corollary}\label{cor:continuity}
		Let $\Lambda$ be a horseshoe for the map $\varphi:\Lambda\to\Lambda$ and $f:\Lambda\to \mathbb{R}$ be a function of the form $f(x)=\max_{1\leq j\leq N}f_j(x)$, where each $f_j$ is a smooth function defined in a neighborhood of $\Lambda$ such that for every $x\in \Lambda$ $Df_j(x)(e^{s,u}_x)\neq 0$. Then \Cref{thm:main} holds for $(\varphi,\Lambda,f)$.
	\end{corollary}

	\section{The second Lagrange spectra}
	
	\subsection{Deduction of the formulas for $k_2$ and $k_2^*$}\label{chap:second_spectra_formulas}
	
	In order to give a dynamical characterization for $\mathbb{L}_2$ and $\mathbb{L}_2^*$ we will, for a given $\alpha \in \mathbb{R}\setminus \mathbb{Q}$, deduce formulas for $$k_2(\alpha)=\underset{(p,q)\neq(p_{n},q_{n}),\forall n\in \mathbb{N}}{\underset{|p|,|q|\rightarrow \infty}{\limsup}}|q(q\alpha-p)|^{-1}$$ and $$k_2^*(\alpha)=\underset{\frac{p}{q}\neq\frac{p_{n}}{q_{n}},\forall n\in \mathbb{N}}{\underset{|p|,|q|\rightarrow \infty}{\limsup}}|q(q\alpha-p)|^{-1},$$ where $\frac{p_n}{q_n}$ are the convergents of the continued fraction of $\alpha=[0;a_1,a_2,\cdots]$. The formula for $k_2$ is the same as in Moshchevitin's paper. The expression of the formula for $k_2^*$ given below will be slightly different than the one given by Moshchevitin, but the formulas will be equivalent. The method of proof for both formulas is, however, different.

	\begin{lemma}
		If $a_{n}=1$, then
		\begin{equation*}
			\frac{\alpha_n+\beta_n}{(2\alpha_n-1)(2\beta_n+1)}=\frac{\alpha_{n+1}+\beta_{n+1}}{(2-\beta_{n+1})(\alpha_{n+1}+2)}.
		\end{equation*}
	\end{lemma}
	
	\begin{proof}
		When $a_{n}=1$, it holds that $\alpha_n=1+\frac{1}{\alpha_{n+1}}$ and $\beta_n=\frac{1}{\beta_{n+1}}-1$. Therefore, 
		\begin{align*}
			\frac{\alpha_n+\beta_n}{(2\alpha_n-1)(2\beta_n+1)}&=\frac{(1+\frac{1}{\alpha_{n+1}})+(-1+\frac{1}{\beta_{n+1}})}{(\frac{2}{\beta_{n+1}}-1)(1+\frac{2}{\alpha_{n+1}})} \\
			&=\frac{(\frac{1}{\alpha_{n+1}}+\frac{1}{\beta_{n+1}})}{(\frac{2}{\beta_{n+1}}-1)(1+\frac{2}{\alpha_{n+1}})} \\
			&=\frac{(\frac{\alpha_{n+1}+\beta_{n+1}}{\alpha_{n+1}\beta_{n+1}})}{(\frac{2-\beta_{n+1}}{\beta_{n+1}})(\frac{\alpha_{n+1}+2}{\alpha_{n+1}})} \\
			&=\frac{\alpha_{n+1}+\beta_{n+1}}{(2-\beta_{n+1})(\alpha_{n+1}+2)}.
		\end{align*}
	\end{proof}
	
	\begin{lemma}
		Let $\alpha\in\R\setminus\Q$, such that $\alpha=[a_0;a_1,a_2,\dots]$. Let $p\in \mathbb{Z}$, $q\in \mathbb{N}_{>0}$ such that $(p,q)\neq (p_{n},q_{n})$, $\forall n\in \mathbb{N}$. Let $k\in \mathbb{N}$ such that $q_{k}\leq q < q_{k+1}$ and suppose that 
		\begin{equation*}
			|q(q\alpha-p)|^{-1} = \max\{|s(s\alpha-t)|^{-1},t\in\Z,s\in\N\cap[q_k,q_{k+1})\}.
		\end{equation*}
		Then there is a $k$ such that either
		\begin{enumerate}
			\item
			\begin{equation*}
				(p,q)=(2p_k,2q_k) \quad\text{and}\quad |q(q\alpha-p)|^{-1}= \frac{\alpha_{k+1}+\beta_{k+1}}{4}
			\end{equation*}
			\item
			\begin{equation*}
				(p,q)=(2p_{k-1},2q_{k-1}) \quad\text{and}\quad |q(q\alpha-p)|^{-1}= \frac{\alpha_{k}+\beta_{k}}{4}
			\end{equation*}
			\item
			\begin{equation*}
				\frac{p}{q}=\frac{p_k+p_{k-1}}{q_k+q_{k-1}}, \quad\text{and}\quad \abs{q(q\alpha-p)}^{-1} =\frac{\alpha_{k+1}+\beta_{k+1}}{(1+\beta_{k+1})(\alpha_{k+1}-1)}
			\end{equation*}
			\item 
			\begin{equation*}
				\frac{p}{q}=\frac{2p_{k}-p_{k-1}}{2q_{k}-q_{k-1}}, \quad\text{and}\quad |q(q\alpha-p)|^{-1}=\frac{\alpha_{k+1}+\beta_{k+1}}{(\alpha_{k+1}+2)(2-\beta_{k+1})}.
			\end{equation*}
			\item 
			\begin{equation*}
				\frac{p}{q}=\frac{(a_{k+1}-1)p_{k}+p_{k-1}}{(a_{k+1}-1)q_{k}+q_{k-1}}, \quad\text{and}\quad |q(q\alpha-p)|^{-1}=\frac{\alpha_{k+1}+\beta_{k+1}}{(\alpha_{k+1}-a_{k+1}+1)(a_{k+1}-1+\beta_{k+1})}
			\end{equation*}
		\end{enumerate}
	\end{lemma}
	
	\begin{proof}
		
		Let $k\in \mathbb{N}$ such that $q_{k}\leq q < q_{k+1}$. We would like to see what are the possible values for $(p,q)$ that can correspond to the best approximation of $\alpha$ by non convergents with $q$ in this range.
		
		If $k$ is even, we have $\frac{p_{k}}{q_{k}}< \alpha <\frac{p_{k+1}}{q_{k+1}}$, and if $k$ is odd, we have $\frac{p_{k}}{q_{k}}>\alpha>\frac{p_{k+1}}{q_{k+1}}$. We will assume without loss of generality that $k$ is even.
		
		As was seen before, $$\abs{\alpha-\frac{p_{k}}{q_{k}}}=\frac{1}{(\alpha_{k+1}+\beta_{k+1})q_{k}^{2}},$$ where $\alpha_{k+1}=[a_{k+1};a_{k+2},a_{k+3},...]$ and $\beta_{k+1}=[0;a_{k},a_{k-1},...,a_{1}]=\frac{q_{k-1}}{q_{k}}$.
		
		The first family of $(p,q)$ to consider are the ones such that $\frac{p}{q}=\frac{p_{k}}{q_{k}}$. In this case we have $(p,q)=(np_{k},nq_{k})$ with $n\geq 2$, so 
		\begin{equation*}
			|q(q\alpha-p)|^{-1}=\frac{\alpha_{k+1}+\beta_{k+1}}{n}\leq \frac{\alpha_{k+1}+\beta_{k+1}}{4},
		\end{equation*}
		with equality for $(p,q)=(2p_{k},2q_{k})$.
		
		Suppose now that $\frac{p}{q}\neq \frac{p_{k}}{q_{k}}$. Since we are assuming that $k$ is even, we have $\frac{p_{k}}{q_{k}}< \alpha <\frac{p_{k+1}}{q_{k+1}}$, and $|\frac{p_{k}}{q_{k}}-\frac{p}{q}|=\frac{|qp_{k}-pq_{k}|}{qq_{k}}\geq \frac{1}{qq_{k}}>\frac{1}{q_{k+1}q_{k}}=|\frac{p_{k+1}}{q_{k+1}}-\frac{p_{k}}{q_{k}}|$. This implies that $\frac{p}{q}$ is outside the interval $[\frac{p_{k}}{q_{k}},\frac{p_{k+1}}{q_{k+1}}]$.
		
		Let us assume that $a_{k+1}\geq 2$ for infinitely many positive integers $k$. For those values of $k$, consider the approximation $\frac{p}{q}=\frac{p_{k}+p_{k-1}}{q_{k}+q_{k-1}}$ of $\alpha$. We have $q_{k}<q=q_{k}+q_{k-1}<q_{k+1}$.
		
		We have 
		\begin{equation*}
			\begin{split}
				\abs{\alpha-\frac{p}{q}} &=\abs{\alpha-\frac{p_{k}+p_{k-1}}{q_{k}+q_{k-1}}}\\
				&=\abs{\frac{\alpha_{k+1}p_{k}+p_{k-1}}{\alpha_{k+1}q_{k}+q_{k-1}}-\frac{p_{k}+p_{k-1}}{q_{k}+q_{k-1}}}\\
				&=\abs{\frac{(\alpha_{k+1}-1)(p_{k}q_{k-1}-p_{k-1}q_{k})}{(\alpha_{k+1}q_{k}+q_{k-1})(q_{k}+q_{k-1})}}\\
				&=\abs{\frac{(\alpha_{k+1}-1)(-1)^{k-1}}{(\alpha_{k+1}q_{k}+q_{k-1})(q_{k}+q_{k-1})}}\\
				&=\frac{\alpha_{k+1}-1}{(\alpha_{k+1}q_{k}+q_{k-1})(q_{k}+q_{k-1})}
			\end{split}
		\end{equation*}
		
		and
		
		\begin{equation*}
			\begin{split}
				\abs{q(q\alpha-p)}^{-1} &=\abs{q^{2}(\alpha-\frac{p}{q})}^{-1}\\
				&=\abs{(q_{k}+q_{k-1})^{2}(\alpha-\frac{p_{k}+p_{k-1}}{q_{k}+q_{k-1}})}^{-1}\\
				&=\frac{\alpha_{k+1}q_{k}+q_{k-1}}{(q_{k}+q_{k-1})(\alpha_{k+1}-1)}\\
				&=\frac{\alpha_{k+1}+\beta_{k+1}}{(1+\beta_{k+1})(\alpha_{k+1}-1)}
			\end{split}
		\end{equation*}
		
		As $0<\beta_{k+1}<1$, one has $|q(q\alpha-p)|^{-1}>\frac{1}{1+\beta_{k+1}}>\frac{1}{2}$
		
		Since $\frac{p}{q}$ is outside the interval $[\frac{p_{k}}{q_{k}},\frac{p_{k-1}}{q_{k-1}}]$, we have two cases. If $\frac{p}{q}<\frac{p_{k}}{q_{k}}$, then $$\abs{\alpha-\frac{p}{q}}>\abs{\frac{p_{k}}{q_{k}}-\frac{p}{q}}=\frac{|qp_{k}-pq_{k}|}{qq_{k}}.$$
		
		If $\abs{qp_{k}-pq_{k}}\geq2$, we would have $$\abs{\alpha-\frac{p}{q}}>\frac{2}{qq_{k}}>\frac{2}{q^{2}},$$
		
		which implies that $$q\abs{q\alpha-p}=q^{2}\abs{\alpha-\frac{p}{q}}>2,$$
		
		and $$|q(q\alpha-p)|^{-1}< \frac{1}{2},$$
		
		so the approximation $\frac{p}{q}$ is worse than the approximation $\frac{p_{k}+p_{k-1}}{q_{k}+q_{k-1}}$, and we don't need to consider it.
		
		Now we treat the case where $$|qp_{k}-pq_{k}|=1.$$
		
		Since $qp_{k}-pq_{k}=qq_{k}(\frac{p_{k}}{q_{k}}-\frac{p}{q})>0$, we have that $qp_{k}-pq_{k}=1=(-1)^{k}=-(q_{k-1}p_{k}-p_{k-1}q_{k})$. Therefore, $p_{k}(q+q_{k-1})=q_{k}(p+p_{k-1})$ and, since $gcd(p_{k},q_{k})=1$, there is a positive integer $b$ such that $q+q_{k-1}=bq_{k}$ and $p+p_{k-1}=bp_{k}$. Since $q=bq_{k}-q_{k-1}>q_{k}$, we should have $b\geq 2$. Thus we have the error term 
		\begin{equation*}
			\begin{split}
				\abs{\alpha-\frac{p}{q}} &=\abs{\alpha-\frac{bp_{k}-p_{k-1}}{bq_{k}-q_{k-1}}}\\
				&=\abs{\frac{\alpha_{k+1}p_{k}+p_{k-1}}{\alpha_{k+1}q_{k}+q_{k-1}}-\frac{bp_{k}-p_{k-1}}{bq_{k}-q_{k-1}}}\\
				&=\frac{\alpha_{k+1}+b}{(\alpha_{k+1}q_{k}+q_{k-1})(bq_{k}-q_{k-1})}
			\end{split}
		\end{equation*}
		
		and then 
		\begin{align*}
			\abs{q(q\alpha-p)}^{-1}&=\abs{q^{2}(\alpha-\frac{p}{q})}^{-1} \\
			&=\abs{(bq_{k}-q_{k-1})^{2}(\alpha-\frac{p}{q})}^{-1}=\frac{\alpha_{k+1}q_{k}+q_{k-1}}{(\alpha_{k+1}+b)(bq_{k}-q_{k-1})}=\frac{\alpha_{k+1}+\beta_{k+1}}{(\alpha_{k+1}+b)(b-\beta_{k+1})}.
		\end{align*}
		$$$$
		
		As the function $h(b)=(\alpha_{k+1}+b)(b-\beta_{k+1})$ is increasing for $b\geq0$,it has minimum value for $b=2$, so the best approximation of this kind is $$\frac{p}{q}=\frac{2p_{k}-p_{k-1}}{2q_{k}-q_{k-1}},$$ which is such that $|q(q\alpha-p)|^{-1}=\frac{\alpha_{k+1}+\beta_{k+1}}{(\alpha_{k+1}+2)(2-\beta_{k+1})}$. Notice that we have $q_{k}<2q_{k}-q_{k-1}<q_{k+1}$. The other possibility is when $$\frac{p}{q}>\frac{p_{k}}{q_{k}}.$$
		
		In this case, 
		\begin{equation*}
			\begin{split}
				\abs{\alpha-\frac{p}{q}}>\abs{\frac{p_{k+1}}{q_{k+1}}-\frac{p}{q}}=\abs{\frac{p_{k}}{q_{k}}-\frac{p}{q}}-\abs{\frac{p_{k}}{q_{k}}-\frac{p_{k+1}}{q_{k+1}}}=\frac{|qp_{k}-pq_{k}|}{qq_{k}}-\frac{1}{q_{k}q_{k+1}}>\frac{|qp_{k}-pq_{k}|-1}{qq_{k}}.
			\end{split}
		\end{equation*}
		
		If $|qp_{k}-pq_{k}|\geq3$, we have $|\alpha-\frac{p}{q}|>\frac{2}{qq_{k}}$, then
		
		$$|q(q\alpha-p)|=q^{2}|\alpha-\frac{p}{q}|>\frac{2q}{q_{k}}\geq2$$
		
		which implies $|q(q\alpha-p)|^{-1}<\frac{1}{2}$, so the approximation $\frac{p}{q}$ is worse than $\frac{p_{k}+p_{k-1}}{q_{k}+q_{k-1}}$, and we don't need to consider it.
		
		If $|qp_{k}-pq_{k}|=2$, since $qp_{k}-pq_{k}=qq_{k}(\frac{p_{k}}{q_{k}}-\frac{p}{q})<0$, we will have $qp_{k}-pq_{k}=-2=2(q_{k-1}p_{k}-p_{k-1}q_{k}),$ so $p_{k}(q-2q_{k-1})=q_{k}(p-2p_{k-1})$, and thus there is an integer $b$ such that $q-2q_{k-1}=bq_{k}$ and $p-2p_{k-1}=bp_{k}$. Since $q\geq q_{k}$, we should have $b\geq0$. If $b=0$, $(p,q)=(2p_{k-1},2q_{k-1})$, and $\frac{p}{q}=\frac{p_{k-1}}{q_{k-1}}$, a case which we already considered, and is relevant for $\mathbb{L}_{2}$, but not for $\mathbb{L}^{*}_{2}$.
		
		Let us consider now the case $b\geq1$.We have $\frac{p}{q}=\frac{bp_{k}+2p_{k-1}}{bq_{k}+2q_{k-1}}$. Since $q=bq_{k}+2q_{k-1}<q_{k+1}=a_{k+1}q_{k}+q_{k-1}$, we should have 
		\begin{equation*}
			\begin{split}
				\abs{\alpha-\frac{p}{q}}=\abs{\frac{\alpha_{k+1}p_{k}+p_{k-1}}{\alpha_{k+1}q_{k}+q_{k-1}}-\frac{bp_{k}+2p_{k-1}}{bq_{k}+2q_{k-1}}}&=\abs{\frac{(2\alpha_{k+1}-b)(p_{k}q_{k-1}-p_{k-1}q_{k})}{(\alpha_{k+1}q_{k}+q_{k-1})(bq_{k}+2q_{k-1})}}\\
				&=\frac{2\alpha_{k+1}-b}{(\alpha_{k+1}q_{k}+q_{k-1})(bq_{k}+2q_{k-1})},
			\end{split}  
		\end{equation*}
		
		and $$|q(q\alpha-p)|^{-1}=\abs{q^{2}(\alpha-\frac{p}{q})}^{-1}=\frac{\alpha_{k+1}q_{k}+q_{k-1}}{(2\alpha_{k+1}-b)(bq_{k}+2q_{k-1})}=\frac{\alpha_{k+1}+\beta_{k+1}}{(2\alpha_{k+1}-b)(b+2\beta_{k+1})}.$$
		
		Since $1\leq b<a_{k+1}$,$2\alpha_{k+1}-b>\alpha_{k+1}$ and $b+2\beta_{k+1}\geq 1+2\beta_{k+1}$, so $$|q(q\alpha-p)|^{-1}<\frac{\alpha_{k+1}+\beta_{k+1}}{\alpha_{k+1}(1+2\beta_{k+1})}<\frac{\alpha_{k+1}+\beta_{k+1}}{(
			\alpha_{k+1}-1)(1+\beta_{k+1})},$$
		
		which just is that of the approximation $\frac{p}{q}=\frac{p_{k}+p_{k-1}}{q_{k}+q_{k-1}}$.
		
		So we don't need to consider approximation in the form of $\frac{bp_{k}+2p_{k-1}}{bq_{k}+2q_{k-1}}$ (this is also true even if we consider cases when $a_{k+1}=1$, since in this case $bq_{k}+2q_{k-1}\geq q_{k}+2q_{k-1}>q_{k+1}$)
		
		Finally, let us consider the cases when $|qp_{k}-pq_{k}|=1$,since $qp_{k}-pq_{k}<0$,we have $qp_{k}-pq_{k}=-1=q_{k-1}p_{k}-p_{k-1}q_{k}$, so $p_{k}(q-q_{k-1})=q_{k}(p-p_{k-1})$, and there is a positive integer $b$ such that $q-q_{k-1}=bq_{k}$ and $p-p_{k-1}=bp_{k}$, so $\frac{p}{q}=\frac{bp_{k}+p_{k-1}}{bq_{k}+q_{k-1}}$.
		
		Since $q=bq_{k}+q_{k-1}<q_{k+1}=a_{k+1}q_{k}+q_{k-1}$, we should have $b<a_{k+1}$,so $b\leq a_{k+1}-1$. We have:
		\begin{align*}
			\abs{\alpha-\frac{p}{q}}=\abs{\frac{\alpha_{k+1}p_{k}+p_{k-1}}{\alpha_{k+1}q_{k}+q_{k-1}}-\frac{bp_{k}+p_{k-1}}{bq_{k}+q_{k-1}}}&=\abs{\frac{(\alpha_{k+1}-b)(p_{k}q_{k-1}-p_{k-1}q_{k})}{(\alpha_{k+1}q_{k}+q_{k-1})(bq_{k}+q_{k-1})}}\\
			&=\frac{\alpha_{k+1}-b}{(\alpha_{k+1}q_{k}+q_{k-1})(bq_{k}+q_{k-1})}
		\end{align*}
		and 
		\begin{equation*}
			|q(q\alpha-p)|^{-1}=|q^{2}(\alpha-\frac{p}{q})|^{-1}=\frac{\alpha_{k+1}q_{k}+q_{k-1}}{(\alpha_{k+1}-b)(bq_{k}+q_{k-1})}=\frac{\alpha_{k+1}+\beta_{k+1}}{(\alpha_{k+1}-b)(b+\beta_{k+1})}.
		\end{equation*}
		
		Since the function $q(b)=(\alpha_{k+1}-b)(b+\beta_{k+1})$ gets its minimum for $b\in [1,\alpha_{k+1}-1]$ when $b=1$ or $b=a_{k+1}-1$. Specifically, the minimum is attained for $b=1$ if $\alpha_{k+1}-a_{k+1}\geq \beta_{k+1}$. Otherwise,it is attained for $b=a_{k+1}-1$.
		
		If we consider cases when $a_{k+1}=1$, we can not consider these approximations, since $q=bq_{k}+q_{k-1}\geq q_{k}+q_{k-1}$ in these cases. Also, if $a_{k+1}=1$, in the case $\frac{p}{q}<\frac{p_k}{q_k}$ aforementioned, we should have $b=2$ since, if $b\geq 3$,$q=bq_{k}-q_{k-1}\geq 3q_{k}-q_{k-1}>2q_{k}>q_{k}+q_{k-1}=q_{k+1}$.
		
	\end{proof}

	Analogously to Moshchevitin paper, let us denote the functions
	\begin{align*}
		\kappa_1(k)&:=\frac{\alpha_{k+1}+\beta_{k+1}}{(\alpha_{k+1}-1)(1+\beta_{k+1})}, \\
		\kappa_2(k)&:=\frac{\alpha_{k+1}+\beta_{k+1}}{(\alpha_{k+1}-a_{k+1}+1)(a_{k+1}-1+\beta_{k+1})}, \\
		\kappa_3(k)&:=\frac{\alpha_{k+1}+\beta_{k+1}}{(2-\beta_{k+1})(\alpha_{k+1}+2)}, \\
		\kappa_4(k)&:=\frac{\alpha_{k+1}+\beta_{k+1}}{4}.
	\end{align*}

	\subsubsection{The case of $\mathbb{L}_{2}$}

	\begin{lemma}
		For $\alpha\in\R\setminus\Q$, we have $k_2(\alpha)=\limsup_{k\to \infty}\delta(k)$, where 
		
		$$\delta(k):=
		\begin{cases}
			\max\left\{\frac{\alpha_{k+1}+\beta_{k+1}}{4},\frac{\alpha_{k+1}+\beta_{k+1}}{(\alpha_{k+1}-1)(1+\beta_{k+1})},\frac{\alpha_{k+1}+\beta_{k+1}}{(\alpha_{k+1}-a_{k+1}+1)(a_{k+1}-1+\beta_{k+1})}\right\}, & \text{if } a_{k+1}\geq2\\
			\frac{\alpha_{k+1}+\beta_{k+1}}{4}, &\text{if } a_{k+1}=1
		\end{cases}
		$$
	\end{lemma}
	
	\begin{proof}
		From now on, we fix $\alpha\in \mathbb{R}\setminus \mathbb{Q}$. We suppose that $\alpha=[a_0;a_1,a_2,\cdots]$ is such that infinitely many $a_k$ are larger or equal than $2$. If this does not happen, then $\alpha$ will be equivalent to $[0;\overline{1}]=\frac{-1+\sqrt{5}}{2}$ and clearly in that case
		\begin{equation*}
			\lim_{k\to\infty}\frac{\alpha_{k+1}+\beta_{k+1}}{4}=\frac{[0;\overline{1}]+[1;\overline{1}]}{4}=\frac{\frac{-1+\sqrt{5}}{2}+\frac{1+\sqrt{5}}{2}}{4}=\frac{\sqrt{5}}{4}.
		\end{equation*}
		
		Since $a_{k+1}\geq2$, for infinitely many positive integers $k$, if $a_{k+1}\geq 3$ for  infinitely many positive integers $k$, then $\alpha_{k+1}+\beta_{k+1}>a_{k+1}\geq3$ for infinitely many values of $k$ and 
		
		$$\underset{(p,q)\neq(p_{n},q_{n}),\forall n\in \mathbb{N}}{\underset{|p|,|q|\rightarrow \infty}{\limsup}}|q(q\alpha-p)|^{-1}\geq \limsup_{k\rightarrow \infty} \kappa_4(k)\geq \frac{3}{4}>\frac{2}{3},$$
		
		and if $a_{k+1}\leq2$ for $k$ large enough, if $a_{k+1}=2$ and $k$ is large, we have $$\alpha_{k+1}+\beta_{k+1}=2+[0;a_{k+2},a_{k+3},...]+[0;a_{k},a_{k-1},...,a_{1}]\geq 2+\frac{1}{3}+\frac{1}{3}=\frac{8}{3},$$
		
		so
		
		$$\underset{(p,q)\neq(p_{n},q_{n}),\forall n\in \mathbb{N}}{\underset{|p|,|q|\rightarrow \infty}{\limsup}}|q(q\alpha-p)|^{-1}\geq \limsup_{k\rightarrow \infty}\kappa_4(k)\geq \frac{1}{4}\cdot\frac{8}{3}=\frac{2}{3}.$$
		
		If $a_{k+1}\geq4$ for infinitely many positive integers $k$, then $\alpha_{k+1}+\beta_{k+1}>a_{k+1}\geq4$ for infinitely many values of $k$ and 
		
		$$\underset{(p,q)\neq(p_{n},q_{n}),\forall n\in \mathbb{N}}{\underset{|p|,|q|\rightarrow \infty}{\limsup}}|q(q\alpha-p)|^{-1}\geq \limsup_{k\rightarrow \infty}\kappa_4(k)\geq \frac{4}{4}=1.$$
		
		For the best approximation in case $\frac{p}{q}<\frac{p_k}{q_k}$, we have $\frac{p}{q}=\frac{2p_{k}-p_{k-1}}{2q_{k}-q_{k-1}}$ and 
		\begin{equation*}
			|q(q\alpha-p)|^{-1}=\frac{\alpha_{k+1}+\beta_{k+1}}{(\alpha_{k+1}+2)(2-\beta_{k+1})}<\frac{\alpha_{k+1}+\beta_{k+1}}{\alpha_{k+1}+2}<1,
		\end{equation*}
		so if $a_{k+1}\geq4$ for infinitely many values of $k$, we don't need to consider these approximations.
		
		If $a_{k+1}\leq3$ for every large $k$, we have $\beta_{k+1}=[0;a_{k},a_{k-1},...,a_{1}]<\frac{4}{5}$, so since $\frac{\alpha+\beta}{2-\beta}$ is increasing in $\beta$, 
		\begin{equation*}
			\frac{\alpha_{k+1}+\beta_{k+1}}{2-\beta_{k+1}}\leq \frac{\alpha_{k+1}+\frac{4}{5}}{\frac{6}{5}} \quad\text{and}\quad \frac{\alpha_{k+1}+\beta_{k+1}}{(\alpha_{k+1}+2)(2-\beta_{k+1})}\leq \frac{5}{6}\frac{\alpha_{k+1}+\frac{4}{5}}{\alpha_{k+1}+2}. 
		\end{equation*}

		Since $\alpha_{k+1}=[a_{k+1};a_{k+2},...]<4$ and $\frac{\alpha+\frac{4}{5}}{\alpha+2}$ is increasing in $\alpha$, $\frac{5}{6}\frac{\alpha_{k+1}+\frac{4}{5}}{\alpha_{k+1}+2}\leq \frac{5}{6}\frac{4+\frac{4}{5}}{6}=\frac{2}{3}$ and thus, since $\kappa_{1}(k)\geq \frac{2}{3}$ for infinitely many values of $k$,we don't need to consider the approximations from case $\frac{p}{q}<\frac{p_k}{q_k}$.
		
		This implies that if $a_{k+1}\geq2$ for infinitely many values of $k$, then $$\underset{(p,q)\neq(p_{n},q_{n}),\forall n\in \mathbb{N}}{\underset{|p|,|q|\rightarrow \infty}{\limsup}}|q(q\alpha-p)|^{-1},$$ can be written into the form of
		
		$$\limsup_{\substack{a_{k+1}\geq2 \\ k\rightarrow \infty}}\max\{\kappa_2(k),\kappa_3(k),\kappa_{4}(k)\}.$$
	\end{proof}

	\subsubsection{The case of $\mathbb{L}^{*}_{2}$}
	
	\begin{lemma}
		For $\alpha\in \mathbb{R}\setminus \mathbb{Q}$, we have $k_2^*(\alpha)=\limsup_{n\to \infty}\gamma(k)$, where
		$$\gamma(k):=
		\begin{cases}
			\max\left\{\frac{\alpha_{k+1}+\beta_{k+1}}{(2-\beta_{k+1})(\alpha_{k+1}+2)},\frac{\alpha_{k+1}+\beta_{k+1}}{(\alpha_{k+1}-1)(1+\beta_{k+1})},\frac{\alpha_{k+1}+\beta_{k+1}}{(\alpha_{k+1}-a_{k+1}+1)(a_{k+1}-1+\beta_{k+1})}\right\}, &\text{if }a_{k+1}\geq2\\
			\frac{\alpha_{k+1}+\beta_{k+1}}{(2-\beta_{k+1})(\alpha_{k+1}+2)}, & \text{if }a_{k+1}=1.
		\end{cases}
		$$
	\end{lemma}
	
	\begin{proof}
		The previous discussion implies that, if $a_{k+1}\geq2$ for infinitely many values of $k$, then $$\underset{\frac{p}{q}\neq\frac{p_n}{q_n},\forall n\in \mathbb{N}}{\underset{|p|,|q|\rightarrow \infty}{\limsup}}|q(q\alpha-p)|^{-1}=\underset{k\rightarrow \infty}{\limsup}\hspace{0.1cm}\gamma(k).$$

		Now we consider the case $a_{k}=1$ for every large $k$.
		
		The approximations with $\frac{p}{q}<\frac{p_k}{q_k}$ give $|q(q\alpha-p)|^{-1}=\frac{\alpha_{k+1}+\beta_{k+1}}{(\alpha_{k+1}+2)(2-\beta_{k+1})}$. Since, for $k$ large, $\alpha_{k+1}=[1;1,1,...]=\frac{\sqrt{5}+1}{2}$, we have $\lim_{k\rightarrow\infty}\frac{\alpha_{k+1}+\beta_{k+1}}{(\alpha_{k+1}+2)(2-\beta_{k+1})}=\frac{\sqrt{5}}{5}$.
		
		In case $\frac{p}{q}<\frac{p_k}{q_k}$, if $|qp_{k}-pq_{k}|\geq2$, we would have $|\alpha-\frac{p}{q}|=|\frac{p_{k}}{q_{k}}-\frac{p}{q}|+|\alpha-\frac{p_{k}}{q_{k}}|\geq\frac{2}{qq_{k}}+|\alpha-\frac{p_{k}}{q_{k}}|$. Since we have $|\alpha-\frac{p_{k}}{q_{k}}|=\frac{1}{(\alpha_{k+1}+\beta_{k+1})q_{k}^{2}}$ and $\lim_{k\rightarrow\infty}(\alpha_{k+1}+\beta_{k+1})=\sqrt{5}<\frac{12}{5}$,we have $|\alpha-\frac{p_{k}}{q_{k}}|>\frac{5}{12q_{k}^{2}}$ for $k$ large, so 
		
		\begin{equation*}
			\begin{split}
				\abs{q(q\alpha-p)}=\abs{q^{2}(\alpha-\frac{p}{q})}\geq\frac{2q}{q_{k}}+\frac{5q^{2}}{12q_{k}^{2}}\geq2+\frac{5}{12}=\frac{29}{12}
			\end{split}
		\end{equation*}
		
		which means that $|q(q\alpha-p)|^{-1}\leq\frac{12}{29}<\frac{\sqrt{5}}{5}$.
		
		In case $\frac{p}{q}>\frac{p_k}{q_k}$, if $|qp_{k}-pq_{k}|\geq3$, we would have $|\alpha-\frac{p}{q}|=|\frac{p}{q}-\frac{p_{k}}{q_{k}}|-|\alpha-\frac{p_{k}}{q_{k}}|\geq \frac{3}{qq_{k}}-|\alpha-\frac{p_{k}}{q_{k}}|$. Again, since $|\alpha-\frac{p_{k}}{q_{k}}|=\frac{1}{(\alpha_{k+1}+\beta_{k+1})q_{k}^{2}}$ and $\lim_{k\rightarrow\infty}(\alpha_{k+1}+\beta_{k+1})=\sqrt{5}$, since $\frac{q}{q_{k}}<\frac{q_{k+1}}{q_{k}}$, $\lim_{k\rightarrow\infty}\frac{q_{k+1}}{q_{k}}=\frac{1+\sqrt{5}}{2}$ and
		\begin{equation*}
			\abs{q\alpha-p}=q\abs{\alpha-\frac{p}{q}}\geq\frac{3}{q_{k}}-\frac{q}{q_{k}}\abs{q_{k}(\alpha-p_{k})}\geq\frac{3}{q_{k}}-\frac{1.62}{2.23q_{k}}>\frac{\sqrt{5}}{q_{k}}\geq\frac{\sqrt{5}}{q},
		\end{equation*}
		we have $q|q\alpha-p|>\sqrt{5}$ and $|q(q\alpha-p)|^{-1}<\frac{\sqrt{5}}{5}$.
		
		This and the previous discussion implies that if $a_{k}=1$ for every large $k$ then
		
		$$\underset{|p|,|q|\rightarrow\infty,(p,q)\neq(p_{n},q_{n}),\forall n\in \mathbb{N}}{\limsup}|q(q\alpha-p)|^{-1}=\max\left\{\frac{\sqrt{5}}{5},\lim_{k\rightarrow\infty}\kappa_4(k)\right\}=\frac{\sqrt{5}}{4}$$
		
		and $\underset{\frac{p}{q}\neq\frac{p_n}{q_n},\forall n\in \mathbb{N}}{\underset{|p|,|q|\rightarrow \infty}{\limsup}}|q(q\alpha-p)|^{-1}=\lim_{k\rightarrow\infty}\frac{\alpha_{k+1}+\beta_{k+1}}{(\alpha_{k+1}+2)(2-\beta_{k+1})}=\frac{\sqrt{5}}{2}$.
		
		This implies, in general, the formula
		
		$$\underset{\frac{p}{q}\neq\frac{p_{n}}{q_{n}},\forall n\in \mathbb{N}}{\underset{|p|,|q|\rightarrow \infty}{\limsup}}|q(q\alpha-p)|^{-1}=\underset{k\rightarrow\infty}{\limsup}\hspace{0.1cm}\gamma(k).$$
		
	\end{proof}
	
	\subsection{Dynamical characterization of $L_2$}

	Recall that Moshchevitin proved that, if $a_{k+1}\geq2$ for infinitely many values of $k$, then $$k_2(\alpha)=\underset{(p,q)\neq(p_{n},q_{n}),\forall n\in \mathbb{N}}{\underset{|p|,|q|\rightarrow \infty}{\limsup}}|q(q\alpha-p)|^{-1}$$ can be written into the form of $k_2(\alpha)=\limsup_{k\to \infty}\delta(k)$, where 
	
	$$\delta(k):=
	\begin{cases}
		\max\left\{\frac{\alpha_{k+1}+\beta_{k+1}}{4},\frac{\alpha_{k+1}+\beta_{k+1}}{(\alpha_{k+1}-1)(1+\beta_{k+1})},\frac{\alpha_{k+1}+\beta_{k+1}}{(\alpha_{k+1}-a_{k+1}+1)(a_{k+1}-1+\beta_{k+1})}\right\}, & \text{if } a_{k+1}\geq2\\
		\frac{\alpha_{k+1}+\beta_{k+1}}{4}, &\text{if } a_{k+1}=1
	\end{cases}
	$$

	\begin{lemma}
		If $k_2(\alpha)\geq2$, then
		\begin{equation}\label{eq:k_2(alpha)greaterthan2}
			k_2(\alpha)=\underset{(p,q)\neq(p_{n},q_{n}),\forall n\in \mathbb{N}}{\underset{|p|,|q|\rightarrow \infty}{\limsup}}|q(q\alpha-p)|^{-1}=\underset{k\rightarrow\infty}{\limsup}\hspace{0.1cm}\frac{\alpha_{k+1}+\beta_{k+1}}{4}
		\end{equation}    
	\end{lemma}
	
	\begin{proof}
		Notice that, since $\frac{\alpha+\beta}{\alpha-1}$ is decreasing in $\alpha$, and $\alpha_{k+1}\geq2$ if $a_{k+1}\geq2$, then $\frac{\alpha_{k+1}+\beta_{k+1}}{\alpha_{k+1}-1}\leq2+\beta_{k+1}$, and thus $\frac{\alpha_{k+1}+\beta_{k+1}}{(\alpha_{k+1}-1)(1+\beta_{k+1})}\leq\frac{2+\beta_{k+1}}{1+\beta_{k+1}}<2$. Also, since $\frac{\alpha+\beta}{\alpha-a_{k+1}}$ is decreasing in $\alpha$, and $\alpha_{k+1}\geq a_{k+1}$, we have
		$$\frac{\alpha_{k+1}+\beta_{k+1}}{\alpha_{k+1}-a_{k+1}+1}\leq a_{k+1}+\beta_{k+1}$$
		therefore 
		$$\frac{\alpha_{k+1}+\beta_{k+1}}{(\alpha_{k+1}-a_{k+1}+1)(a_{k+1}-1+\beta_{k+1})}\leq \frac{a_{k+1}+\beta_{k+1}}{a_{k+1}+\beta_{k+1}-1}<2,$$
		since $a_{k+1}+\beta_{k+1}\geq a_{k+1}\geq2$.
	\end{proof}

	\begin{lemma}
		\begin{equation*}
			\left[\frac{21}{11},\infty\right) \subset L_2.
		\end{equation*}
	\end{lemma}
	\begin{proof}
		Since the classical Lagrange spectrum satisfies
		
		$$L=\left\{\underset{k\rightarrow\infty}{\limsup}\hspace{0.1cm}(\alpha_{k+1}+\beta_{k+1})\right\}\supset[8,+\infty),$$
		
		we have
		$$\left\{\underset{k\rightarrow\infty}{\limsup}\left(\frac{\alpha_{k+1}+\beta_{k+1}}{4}\right)\right\}\supset[2,+\infty).$$

		Together with \eqref{eq:k_2(alpha)greaterthan2}, we then have that $[2,+\infty)\subset L_2$, and if $k_2(\alpha)\leq 2$, then we will have that $a_{k+1}\leq 7$ for every large enough $k$. Now, as $\beta_{k+1}=[0;a_k,\cdots,a_1]$, if $\beta_{k+1}\leq \frac{1}{10}$, then $a_k\geq 10$. Therefore, we actually have that if $a_n\leq 7$ for every $n$ and $a_{k+1}\geq 2$, $\frac{2+\beta_{k+1}}{1+\beta_{k+1}}<\frac{21}{11}<2$ and $\frac{a_{k+1}+\beta_{k+1}}{\alpha_{k+1}+\beta_{k+1}-1}<\frac{21}{11}$. By the same argument, this implies that if $k_2(\alpha)\geq \frac{21}{11}$, then $a_k\leq 4$ so $\frac{\alpha_{k}+\beta_k}{4}\leq \frac{84}{11}<8$ for every large enough $k$, and since $L\supset [\frac{84}{11},+\infty)$ we will also have $L_2\supset [\frac{21}{11},+\infty)$.
		
	\end{proof}
	
	This allows us to reduce the study of $L_{2}$ to the study of a dynamical Lagrange Spectrum associated to a horseshoe.
	
	Let 
	$$C(7)=\{[0;a_{1},a_{2},a_{3},...]\in[0,1]:a_{j}\leq7,\forall j\geq1\}$$ 
	
	and 
	
	$$T:C(7)\times C(7)\rightarrow C(7)\times C(7)$$
	
	given by 
	
	$$T(x,y)=\left(\frac{1}{x}-\left\lfloor\frac{1}{x}\right\rfloor,\frac{1}{y+[\frac{1}{x}]}\right)=\left(\frac{1}{x}-\left\lfloor\frac{1}{x}\right\rfloor,\frac{1}{y+\left\lfloor\frac{1}{x}\right\rfloor}\right)$$
	
	so
	
	$$T([0;a_{0},a_{1},a_{2},...],[0;b_{1},b_{2},...])=([0;a_{1},a_{2},a_{3},...],[0;a_{0},b_{1},b_{2},...]).$$
	
	The map $T$ extends with the same expression to a neighborhood of $\Delta=C(7)\times C(7)$,and $\Delta$ is a horseshoe for this map.
	
	Let $l:\Delta\rightarrow \mathbb{R}$ be given by 
	
	\begin{equation*}
		l(x,y)=
		\begin{cases}
			\max\left\{\frac{y+\frac{1}{x}}{4},\frac{y+\frac{1}{x}}{(\frac{1}{x}-1)(1+y)},\frac{y+\frac{1}{x}}{(\frac{1}{x}-[\frac{1}{x}]+1)([\frac{1}{x}]+y-1)}\right\}, &\text{if } x<\frac{1}{2}, \\
			\frac{y+\frac{1}{x}}{4}, &\text{if } x>\frac{1}{2}.
		\end{cases}
	\end{equation*}
	
	Notice that $x=[0;a_{0},a_{1},...]<\frac{1}{2}$ if and only if $a_{0}\geq2$. Then we have $$L_{2}\cap [0,2]=L_{T,l,\Delta}\cap[0,2],$$
	where $L_{T,l,\Delta}$ denotes the dynamical spectrum associated with the data $(T,l,\Delta)$.

	\subsection{Continuity of dimension for $L_2$}
	
	We will now check that that the data $(T,l,\Delta)$ satisfies the generic hypotheses of the version of the theorem of continuity of dimension for maxima of smooth functions. According to what was mentioned before, the horseshoe $\Delta$ will be in the set $\tilde{U}$ of \Cref{main} because of the fact that the Gauss-Cantor sets are non-essentially affine. Regarding $l$, we see that $l$ can be written in a neighborhood of $\Delta$ as $l(x,y)=\max\{l_1(x,y),l_2(x,y),l_3(x,y)\}$, where $l_1,l_2,l_3$ are in $\tilde{R}_{T,\Delta}$. Set
	
	\begin{eqnarray*}
		l_1(x,y)&=& \frac{y+\frac{1}{x}}{4}\\
		l_2(x,y)&=&\begin{cases}
			\frac{y+\frac{1}{x}}{(\frac{1}{x}-1)(1+y)}, & \text{if } x <\frac{1}{2}\\
			-y-x, \text{if } x>\frac{1}{2}
			
		\end{cases} \\
		l_3(x,y)&=& \begin{cases}
			\frac{y+\frac{1}{x}}{(\frac{1}{x}-[\frac{1}{x}]+1)([\frac{1}{x}]+y-1)}, & \text{ if } x<\frac{1}{2}\\
			-y-x, \text{ if } x > \frac{1}{2}
		\end{cases}
	\end{eqnarray*}
	
	Then we have that $l(x,y)=\max\{l_1(x,y),l_2(x,y),l_3(x,y)\}$, and $l_1,l_2,l_3$ are smooth in $\{(x,y)\in (0,1)\times (0,1)|x\neq \frac{1}{2}\}$. The expression $-y-x$ is just a formality to include this function on the hypotheses of the theorem, but it will never contribute to the value of $l$ since it takes negative values. As $\frac{\partial (-y-x)}{\partial x}=\frac{\partial (-y-x)}{\partial y}=-1$, this function satisfies the generic hypotheses of the theorem.
	
	Notice that for each of the functions $$\frac{y+\frac{1}{x}}{4},\quad \frac{y+\frac{1}{x}}{(\frac{1}{x}-1)(1+y)},\quad \frac{y+\frac{1}{x}}{(\frac{1}{x}-[\frac{1}{x}]+1)([\frac{1}{x}]+y-1)},$$ if we fix $y\in C(7)$, each of them is a nonlinear Möbius map on $x$ without poles a neighborhood of $C(7)$, so then the partial derivative of each of them with respect to $x$ is always nonzero in $C(7)\times C(7)$. Similarly, fixing $x\in C(7)$, we have that $\frac{\partial }{\partial y}(\frac{y+\frac{1}{x}}{4})=\frac{1}{4}$ for every $y$, and the other two functions that define $l$ are nonlinear Möbius maps of $y$ without poles in a neighborhood of $\Delta$. This implies that this setting is in the conditions of \Cref{cor:continuity}, and therefore the function $HD(L_2\cap (-\infty,t))$ is continuous for $t< 2$. On the other hand, as $[\frac{21}{11},+\infty)\subset L_2$, then $HD(L_2\cap (-\infty,t))=1$ for $t\geq \frac{21}{11}$, which assures that this function is also continuous for $t\geq 2$. Finally, as the function $d_2$ is continuous and takes the values 0 and 1 (because of Hall's ray), it must be surjective. This proves \Cref{corE}.

	\subsection{Discontinuity of the dimension function $d_2^*$ in the point $\frac{2}{3}$}
	
	As we have seen, the formula for $k_2^*$ given by Moshchevitin for $k_2^*$ gives a dynamical characterization of $L_2^*$ as the dynamical Lagrange spectrum associated to the function $f^*:\Sigma\to \mathbb{R}$, $$f^*((a_n)_{n\in \mathbb{Z}})=\begin{cases}
		\frac{\alpha_{0}+\beta_{0}^*}{(2\alpha_{0}-1)(2\beta_{0}^*+1)}, & \text{if } a_{0}=1 \\
		\max\left\{\frac{\alpha_{0}+\beta_{0}^*}{(\alpha_{0}-1)(1+\beta_{{0}}^*)},\frac{\alpha_{0}+\beta_{0}^*}{(\alpha_{0}-a_{0}+1)(a_{0}-1+\beta_{0}^*)}\right\}, & \text{if } a_{0}\geq 2 
	\end{cases}$$ on the shift $\sigma:\Sigma\to \Sigma$.

	As was observed by Moshchevitin, the numbers of the form $x=[0;1,1,x_3,1,1,x_6,\cdots]$, with $x_{3n}$ going to infinity are such that $k_2^*(x)=\frac{2}{3}$. 
	With this in mind, we will prove that for every $\varepsilon>0$, there is $k_0$ such that for any $\underline{\theta}=(\cdots,x_{-3},1,1;x_0,1,1,x_3,1,1,x_6,\cdots)$, with $x_{3n}\geq k_0$ for $n\in \mathbb{Z}$, $\ell_{\sigma,f^*}(\underline{\theta})\leq \frac{2}{3}+\varepsilon$. In the opposite direction of what was done in section 2, for $\underline{\theta}=(\cdots,1,1,x_{-3},1,1;x_{0},1,1,x_3,\cdots)$, we have that $$\limsup_{n\to +\infty}f^*(\sigma^n(\underline{\theta}))=\limsup_{k\to +\infty}\max_{0\leq j\leq 2}f^*(\sigma^j(\sigma^{3k}(\underline{\theta})))=\limsup_{k\to +\infty}\hat{f}(\sigma^{3k}(\underline{\theta})),$$ for $\hat{f}=\max_{0\leq j\leq 2}f^*\circ \sigma^j$. For $\underline{\theta}=((a_n)_{n\in \mathbb{Z}})$ of this form, we easily see then that 
	\begin{multline*}
		\hat{f}((a_n)_{n\in \mathbb{Z}})=\max\Bigg\{\frac{\alpha_0+\beta_0^*}{(\alpha_0-1)(\beta_0^*+1)},\frac{\alpha_0+\beta_0^*}{(\alpha_0-a_0+1)(a_0-1+\beta_0^*)}, \\
		\frac{\alpha_1+\beta_1^*}{(2\alpha_1-1)(2\beta_1^*+1)},\frac{\alpha_2+\beta_2^*}{(2\alpha_2-1)(2\beta_2^*+1)}\Bigg\}
	\end{multline*}
	where the expressions for 
	\begin{equation*}
		\alpha_1=[1;1,x_3,\cdots], \beta_1^*=[0;x_0,1,1,\cdots], \alpha_2=[1;x_3,1,1,\cdots],\beta_2^*=[0;1,x_0,1,1,\cdots],
	\end{equation*}
	are the corresponding values of $\alpha_0,\beta_0^*$ but for $\sigma(\theta) \text{ and } \sigma^2(\theta)$, respectively. To see that $\hat{f}(\underline{\theta})$ stays close to $\frac{2}{3}$, we first state some useful inequalities, which hold if $k_0$ is large enough:
	
	\begin{itemize}
		\item As $x_0\leq \alpha_0\leq x_0+1$ and $\frac{x_{-3}+2}{2x_{-3}+3}\leq \beta_0^* \leq \frac{x_{-3}+1}{2x_{-3}+1}$, we have that 
		\begin{align*}
			\frac{1}{x_0+1}+\frac{2}{3}-\frac{1}{9x_{-3}+6}&=\frac{1}{x_0+1}+\frac{2x_{-3}+1}{3x_{-3}+2} \\
			&\leq \frac{\alpha_0+\beta_0^*}{(\alpha_0-1)(\beta_0^*+1)} \\
			&=\frac{1}{\alpha_0-1}+\frac{1}{\beta_0^*+1} \\
			&\leq \frac{1}{x_0}+\frac{2x_{-3}+3}{3x_{-3}+5}=\frac{1}{x_0}+\frac{2}{3}-\frac{1}{9x_{-3}+15}
		\end{align*}
		Therefore, we have that $\frac{\alpha_0+\beta_0^*}{(\alpha_0-1)(\beta_0^*+1)}=\frac{1}{x_0}+\frac{2}{3}-\frac{1}{9x_{-3}}+\eta$, with $|\eta|<\frac{1}{9x_{-3}^2}+\frac{1}{x_0^2}$. We also have the simple upper bound $\frac{\alpha_0+\beta_0^*}{(\alpha_0-1)(\beta_0^*+1)}<\frac{2}{3}+\frac{1}{x_0}$.
		
		\item Considering the symmetry of $(\alpha_0,\beta_0^*)$ and $(\beta_0^*+a_0,\alpha_0-a_0)$, the same argument in the previous item shows that $\frac{\alpha_0+\beta_0^*}{(\alpha_0-a_0+1)(a_0-1+\beta_0^*)}=\frac{1}{x_0}+\frac{2}{3}-\frac{1}{9x_3}+\eta', |\eta'|<\frac{1}{9x_3^2}+\frac{1}{x_0^2}$, and also $\frac{\alpha_0+\beta_0^*}{(\alpha_0-a_0+1)(a_0-1+\beta_0^*)}<\frac{2}{3}+\frac{1}{x_0}$ 
		
		\item As $\frac{\alpha+\beta}{(2\alpha-1)(2\beta+1)}=\frac{1}{2}\left(\frac{1}{2\alpha-1}+\frac{1}{2\beta+1} \right)$ and as $\frac{2x_3+1}{x_3+1}\leq\alpha_1=[1;1,x_3,\cdots]\leq \frac{2x_3+3}{x_3+2},\frac{1}{x_0+1}\leq\beta_1^*=[0;x_0,1,1,\cdots]<\frac{1}{x_0},\frac{x_3+2}{x_3+1}\leq \alpha_2=[1;x_3,1,1,\cdots]<\frac{x_3+1}{x_3},\frac{x_0+1}{x_0+2}\leq \beta_2^*=[0;1,x_0,1,1,\cdots]\leq \frac{x_0}{x_0+1}$, we have that 
		\begin{align*}
			\frac{\alpha_1+\beta_1^*}{(2\alpha_1-1)(2\beta_1^*+1)}&\leq \frac{1}{2}\left(\frac{x_0+1}{x_0+3}+\frac{x_3+1}{3x_3+1}\right) \\
			&=\frac{1}{2}\left(1-\frac{2}{x_0+3}+\frac{1}{3}+\frac{2}{9x_3+3} \right)<\left(\frac{2}{3}-\frac{1}{x_0}+\frac{1}{9x_3}\right)
		\end{align*}
		and $$\frac{\alpha_2+\beta_2^*}{(2\alpha_2-1)(2\beta_2^*+1)}\leq \frac{1}{2}\left(\frac{x_3+1}{2x_3+3}+\frac{x_0+1}{3x_0+5}\right)<\frac{5}{12}.$$
	\end{itemize}
	
	This implies that, if $x_{3n}\geq k_0$ for every $n$, except for $\frac{\alpha_2+\beta_2^*}{(2\alpha_1-1)(2\beta_1^*+1)}$ which is asymptotic to $\frac{5}{12}$, the expressions used in the definition of $\hat{f}$ are asymptotic to $\frac{2}{3}$.

	If we consider the subshift $S(k)$ as the subshift of $\Sigma$ generated by the words $(1,1,x), k^2\leq x\leq k^3$, the Lagrange values of elements $\theta\in S(k)$ will be close to $\frac{2}{3}$. We can also consider an accelerated version of this subshift, the map $\sigma^3:\Xi(k)\to \Xi(k)$ over the set $\Xi(k)=\{(\cdots,x_{-3},1,1;x_0,1,1,x_3,1,1,\cdots)|k^2\leq x_{3n}\leq k^3\}\subset S(k)$. The dynamics of $(\sigma^3,\Xi(k))$ is topologically conjugated to the complete subshift on the alphabet $A=\{(1,1,x)|k^2\leq x\leq k^3\}$. 
	
	Summarizing what we have just done in those terms, for every $\varepsilon>0$, there is $k_0\in \mathbb{N}$ such that for every $k\geq k_0$, $L_{\sigma^3,\hat{f},\Xi(k)}\subset L_2^*\cap(-\infty,\frac{2}{3}+\varepsilon)$.
	
	Now, if we define for $k\in \mathbb{N}$ the Gauss-Cantor set $X(k)=\{[0;1,1,x_3,1,1,x_6,\cdots]|k^2\leq x_{3n}\leq k^3\}$ and $\Lambda(k)=G^2(X(k))\times X(k)$, then $\Lambda(k)$ is a conservative horseshoe for $T^3$, with $HD(\Lambda(k))=2HD(X(k))$. Considering the function $\tilde{f}:\Lambda(k)\to \mathbb{R}$ defined by 
	\begin{align*}
		&\tilde{f}(x,y)=\\
		&\max\Bigg\{\frac{\frac{1}{x}+y}{(\frac{1}{x}-1)(y+1)},\frac{\frac{1}{x}+y}{(\frac{1}{x}-\lfloor \frac{1}{x}\rfloor+1)(\lfloor \frac{1}{x}\rfloor-1+y)},\frac{\frac{1}{x_1}+y_1}{(\frac{2}{x_1}-1)(2y_1+1)},\frac{\frac{1}{x_2}+y_2}{(2\frac{1}{x_2}-1)(2y_2+1)}\Bigg\},
	\end{align*}
	where for $x=[0;a_0,a_1,a_2,\cdots]$ and $y=[0;a_{-1},a_{-2},a_{-3},\cdots]$, $x_1=[0;a_{1},a_2,a_3\cdots],y_2=[0;a_0,a_{-1},a_{-2}\cdots],x_2=[0;a_2,a_3,a_4\cdots],y_2=[0;a_1,a_0,a_{-1},\cdots]$, we have that $L_{T^3,\tilde{f},\Lambda(k)}=L_{\sigma^3,f^*,\Xi(k)}\subset L_2^*\cap(-\infty,\frac{2}{3}+\varepsilon)$. As $(x_j,y_j)=T^j(x,y)$ for $j=1,2$ and $DT$ preserves the coordinate axes and as each of the rational functions used in the definition of $f$ are nonlinear Möbius maps when either $x\in G^2(X(k))$ or $y\in X(k)$ is fixed, the partial derivatives of each of them in the coordinate directions does not vanish for $(x,y)\in \Lambda(k)$. Therefore, the data $T^3,f,\Lambda(k)$ satisfies the hypotheses of \Cref{cor:continuity}.
	
	To conclude the proof, we just need the following:
	\begin{lemma}
		For $k\in \mathbb{N}$ large enough, $HD(X(k))>\frac{1}{2}$.
	\end{lemma}

	The proof of this lemma will be done in the next section. This lemma implies that $HD(\Lambda(k))>1$. Therefore, we have that for every $\varepsilon$, 
	\begin{align*}
		d_2^*\left(\frac{2}{3}+\varepsilon\right)=HD(L_2^*\cap (-\infty,\frac{2}{3}+\varepsilon))&\geq HD(L_{\sigma^3,\hat{f},\Xi(k)})\\
		&=HD(L_{T^3,\tilde{f},\Lambda(k)})=\min\{1,HD(\Lambda(k))\}=1.
	\end{align*}

	The discontinuity of the dimension function in this case goes in opposition to \Cref{thm:main} and to the case of $L_2$. Let us clarify why the results are different. What we did was to construct a smooth by parts function $\hat{f}$ and, for every $\varepsilon>0$, a horseshoe $\Lambda=\Lambda(k)$ for the map $T^3$ such that $HD(\Lambda(k))>1$ and $L_{T^3,\tilde{f},\Lambda(k)}\subset L_2^*\cap [-\infty,\frac{2}{3}+\varepsilon)$. However, what is not true is that for some fixed horseshoe $\Lambda$ and some $\varepsilon>0$ we have $L_2^*\cap (-\infty,\frac{2}{3}+\varepsilon)=L_{T^3,\tilde{f},\Lambda}$, because otherwise the function $d_2^*$ would have to be continuous. We can readily see that all the $\Lambda(k)$ for $k\geq k_0$ cannot be all contained in a same horseshoe defined by $T$, because the associated shift would need to have an infinite number of symbols, and therefore would not be compact.
	
	We can see more explicitly how to construct subsets of $L_2^*\cap (-\infty,\frac{2}{3}+\varepsilon)$ of Hausdorff dimension 1 by a proof similar to \cite[Lemma 3]{M3}. Given $\varepsilon>0$, let $k_0$ be large enough such that, for every $\underline{\theta}=(\cdots,1,1,x_{-3},1,1;x_0,1,1,x_3,\cdots)$ with $x_{3n}\geq k_0$, $\hat{f}(\underline{\theta})\leq \frac{2}{3}+\varepsilon$. Consider one such $\underline{\theta}$ such that $x_{-3}=2k_1,x_0=k_0,x_3=k_1$, where $k_1=\lfloor k_0^{3/2}\rfloor$ and all other positions are greater or equal to $k_0^2$. By the earlier estimates, we have that 
	\begin{align*}
		\frac{\alpha_1+\beta_1^*}{(2\alpha_1-1)(2\beta_1^*+1)}&< \frac{2}{3}-\frac{1}{k_0}+\frac{1}{9k_1}<\frac{2}{3}<\frac{2}{3}+\frac{1}{k_0}-\frac{1}{9k_1}+\eta' \\
		&=\frac{\alpha_0+\beta_0^*}{(a_0-1+\beta_0^*)(\alpha_0-a_0+1)} \\
		&<\frac{2}{3}+\frac{1}{k_0}-\frac{1}{18k_1}+\eta=\frac{\alpha_0+\beta_0^*}{(\alpha_0-1)(\beta_0^*+1)}
	\end{align*}
	for $|\eta|<\frac{1}{k_0^2}+\frac{1}{36k_1^2}, |\eta'|<\frac{1}{k_0^2}+\frac{1}{9k_1^2}$, so 
	\begin{align*}
		\hat{f}(\underline{\theta})=\frac{\alpha_0+\beta_0^*}{(\alpha_0-1)(\beta_0^*+1)}&\in \left[\frac{2}{3}+\frac{1}{k_0}-\frac{1}{18k_1}-\frac{1}{k_0^2}-\frac{1}{36k_1^2},\frac{2}{3}+\frac{1}{k_0}-\frac{1}{18k_1}+\frac{1}{k_0^2}+\frac{1}{36k_1^2}\right]\\
		&\subset \left[\frac{2}{3}+\frac{1}{2k_0}-\frac{1}{9k_1},\frac{2}{3}+\frac{1}{2k_0}+\frac{1}{9k_1}\right]
	\end{align*}
	Regarding now $\sigma^3(\underline{\theta})$ and $\sigma^{-3}(\underline{\theta})$, we have that 
	$$\frac{\alpha_4+\beta_4^*}{(2\alpha_4-1)(2\beta_4^*+1)}<\frac{2}{3}-\frac{1}{k_1}+\frac{1}{9x_6}<\frac{2}{3},$$ $$\frac{\alpha_{-1}+\beta_{-1}^*}{(2\alpha_{-1}-1)(2\beta_{-1}^*+1)}<\frac{2}{3}-\frac{1}{2k_1}+\frac{1}{9k_0}<\frac{2}{3}+\frac{1}{2k_0}-\frac{1}{9k_1}<\frac{\alpha_0+\beta_0^*}{(\alpha_0-1)(\beta_0^*+1)},$$ and $$\frac{\alpha_3+\beta_3^*}{(\alpha_3-1)(\beta_3^*+1)},\frac{\alpha_3+\beta_3^*}{(a_3-1+\beta_3^*)(\alpha_3-a_3+1)},\frac{\alpha_{-3}+\beta_{-3}^*}{(\alpha_{-3}-1)(\beta_{-3}^*+1)},\frac{\alpha_{-3}+\beta_{-3}^*}{(a_{-3}-1+\beta_{-3}^*)(\alpha_{-3}-a_{-3}+1)}$$ are smaller than $$\frac{2}{3}+\frac{1}{k_1}<\frac{2}{3}+\frac{1}{2k_0}-\frac{1}{9k_1}<\frac{\alpha_0+\beta_0^*}{(\alpha_0-1)(\beta_0^*+1)}.$$ For all other positions, the terms $x_{3n_0-3},x_{3n_0},x_{3n_0+3}$ will be at least $k_1=\lfloor k_0^{3/2}\rfloor $, so all the corresponding expressions will also be smaller than $\frac{\alpha_0+\beta_0^*}{(\alpha_0-1)(\beta_0^*+1)}$. Therefore, we have that for any $\underline{\theta}=(\cdots,1,1,x_{-3},1,1;x_0,1,1,x_3,\cdots)$ with $x_{3n}\geq k_0$ with $x_{-3}=k_1,x_0=k_0,x_3=2k_1,k_1=\lfloor k_0^{3/2}\rfloor$ and $x_{3n}\geq k_0^2$ for $|n|>1$, $\sup_{n\in \mathbb{Z}}\hat{f}(\sigma^{3n}(\underline{\theta}))=\hat{f}(\underline{\theta})=\frac{\alpha_0+\beta_0^*}{(\alpha_0-1)(\beta_0^*+1)}$. With the same proof, given such $\underline{\theta}$, considering $$\underline{\theta}'= (\cdots,\tau_{n},\tau_{n-1},\cdots,\tau_1;\tau_0,\tau_1,\cdots,\tau_n,\cdots),$$ where $$\tau_n=(x_{-3(n+1)},1,1,x_{-3n},1,1,\cdots,x_{-6},1,1,k_1,1,1,k_0,1,1,2k_1,1,1,x_6,\cdots,x_{3n},1,1,x_{3(n+1)},1,1),$$ the positions $n$ where $\limsup_{n\to \infty}\hat{f}(\sigma^{3n}(\underline{\theta}'))$ is realized are the positions of each block $\tau_n$ corresponding to $k_0$, and any sequence $\hat{\tau_n}\in \Sigma$ with central positions coinciding with $\tau_n$ having the value $k_0$ in the zero position is such that $\hat{\tau_n}\to \underline{\theta}$. Therefore, $\ell_{\sigma^3,\hat{f}}(\underline{\theta}')=\limsup_{n\to \infty}\hat{f}(\sigma^{3n}(\underline{\theta}'))=\hat{f}(\underline{\theta})=\frac{\alpha_0+\beta_0^*}{(\alpha_0-1)(\beta_0^*+1)}=\frac{1}{\alpha_0-1}+\frac{1}{\beta_0^*+1}$. This means that, for a large enough $k_0$, considering the regular Cantor set $X(k_0)$  as defined before and the embeddings $h,g:X(k_0)\to \mathbb{R}$ defined by $h(x)=[0;k_0-1,1,1,k_1,1,1,x_3,1,1,x_6,\cdots]$ and $g(x)=[0;1,1,1,2k_1,1,1,x_6,\cdots]$ for $x=[0;1,1,x_3,1,1,x_6,\cdots]$, then $h(X(k_0))+g(X(k_0))\subset L_2^*\cap[\frac{2}{3},\frac{2}{3}+\frac{1}{k_0})$. As $h,g$ are diffeomorphisms of $X(k)$ with its images, by the dimension formula for arithmetic sums of regular Cantor sets, we have that 
	\begin{align*}
		HD(h(X(k_0))+g(X(k_0)))&=\min\{1,h(HD(X(k_0)))+g(HD(X(k_0)))\} \\
		&=\min\{1,2HD(X(k_0))\}=1
	\end{align*}
	
	\subsection{Structure of $L_2^*$ near $\frac{2}{3}$}
	
	Until now, we have seen that the numbers $\alpha$ equivalent to elements of $X(k)$ for large $k$ are such that $k_2^*(\alpha)\in [\frac{2}{3},\frac{2}{3}+\frac{1}{k}]$. We will now show a converse for that statement: that if an irrational number $\alpha$ has $k_{2}^{*}$ value very close to $\frac{2}{3}$ (belonging to $[\frac{2}{3},\frac{2}{3}+\delta]$ for a suitable small $\delta >0$), then $\alpha$ is equivalent to a number of the form $[0;1,1,x_{3},1,1,x_6,\cdots]$, where $x_{3n}\geq 22$ for all $n\geq 1$. For this, we use the other formula for $k_2^*(\alpha)$.
	
	\begin{lemma}
		If $\alpha\in\R\setminus\Q$ is such that $k_2^{*}(\alpha)\in[\frac{2}{3},\frac{\sqrt{493}}{33})$. Then $\alpha$ is equivalent to $[0;1,1,x_{3},1,1,x_6,\cdots]$ where $x_{3n}\geq 145$ for all $n\geq 1$. Moreover, these constants are optimal because 
		\begin{equation*}
			k_2^{*}([0;\overline{1,1,1,1,21}])=\frac{\sqrt{493}}{33}=0.67283646\dots
		\end{equation*}
		\begin{equation*}
			k_2^*([0;\overline{1,1,145}])=\frac{\sqrt{21317}}{217}=0.67282684\dots<\frac{\sqrt{493}}{33}.
		\end{equation*}
	\end{lemma}
	
	\begin{proof}
		Suppose that for $\alpha_0=[a_0;a_1,\cdots]$ the $k_2^*(\alpha)<\frac{2}{3-\epsilon}$, for some small $\epsilon>0$. Then, if $m\in\N$ is such that $a_{m}>1$, we should have, for $\alpha=\alpha_{m}=[a_m;a_{m+1},\cdots], \beta=\beta_{m}=[0;a_{m-1},\cdots,a_1]$, the inequalities 
		
		$$\frac{\alpha+\beta}{(\alpha+2)(2-\beta)}<\frac{2}{3-\epsilon},$$
		
		and
		
		$$\frac{\alpha+\beta}{(\alpha-1)(1+\beta)}<\frac{2}{3-\epsilon},$$ provided $m$ is large. 
		
		The first inequality implies $(3-\epsilon)\alpha+(3-\epsilon)\beta<8+4\alpha-4\beta-2\alpha\beta$,so $(2\alpha+7-\epsilon)\beta<(1+\epsilon)\alpha+8$, and thus $$\beta<\frac{(1+\epsilon)\alpha+8}{2\alpha+7-\epsilon}$$
		and the second inequality implies $(3-\epsilon)\alpha+(3-\epsilon)\beta<2\alpha-2\beta-2+2\alpha\beta$, so $(2\alpha-(5-\epsilon))\beta>(1-\epsilon)\alpha+2$ and thus
		$$\beta>\frac{(1-\epsilon)\alpha+2}{2\alpha-(5-\epsilon)}.$$

		This implies 
		$$\frac{(1-\epsilon)\alpha+8}{2\alpha-(5-\epsilon)}<\frac{(1+\epsilon)\alpha+2}{2\alpha+7-\epsilon},$$
		so 
		$$\epsilon(4\alpha^{2}+4\alpha+10)>54,$$
		so
		$$\epsilon> \frac{54}{4\alpha^{2}+4\alpha+10}=\frac{54}{(2\alpha+1)^{2}+9},$$
		and
		$$2\alpha+1>\sqrt{\frac{54}{\epsilon}-9}.$$

		If we take $\epsilon=\frac{3}{113}$,this gives $2\alpha+1>\sqrt{18\times113-9}=45$, and thus $\alpha>22$, so $a_{m}=\lfloor\alpha\rfloor\geq22$. Moreover, since 
		$$\beta<\frac{(1+\epsilon)\alpha+8}{2\alpha+7-\epsilon}$$
		which is decreasing in $\alpha$, we have 
		$$\beta<\frac{(1+\frac{3}{113})\times22+8}{44+7-\frac{3}{113}}=\frac{3}{5}=[0;1,1,2]$$
		which implies that the continued fraction of $\beta$ starts with $[0;1,1,k,...]$, with $k>1$.
		
		If we call a large coefficient any coefficient $a_{m}\geq2$, this implies that, for $m$ large, if $a_{m}$ is large than $a_{m}\geq22,a_{m-1}=1,a_{m-2}=1$ and $a_{m-3}$ is large, which implies the structure described in the beginning, provided that $k_{2}^{*}(\alpha)
		<\frac{2}{3-\frac{3}{113}}=\frac{113}{168}=0.672619...$ and the continued fraction of $\alpha$ does not end with infinitely many coefficients equal to 1 (in this case, the $k_2^*(\alpha)$ would be $\frac{\sqrt{5}}{5}<\frac{1}{2}<\frac{2}{3}$).
		
		It is possible to determine the smallest element no less than $\frac{2}{3}$ of $L_2^*$,which is the $k_2^*$ value of an irrational number $\alpha$ whose continued fraction representation does not have the structure describe above: it is $\frac{\sqrt{493}}{33}=0.67283646...=k([0;\overline{1,1,1,1,21}])$.
		
		Notice that $\frac{\sqrt{493}}{33}=\frac{2}{3-\epsilon}$, with $\epsilon=0.027509...$, so as before, if the $\mathbb{L}_{2}^{*}$ value of $\alpha$ is $\leq \frac{\sqrt{493}}{33}$ then if $m$ is large and $a_{m}\geq2$, we should have
		$$2\alpha+1=2\alpha_{m}+1\geq \sqrt{\frac{54}{\epsilon}-9}>\sqrt{\frac{54}{0.0276}-9}>44.13,$$ so $\alpha=\alpha_{m}>21.565$, and in particular, $a_{m}\geq21$.
		
		Now, if $m$ is large and $a_{m}\geq2$, let $n\geq1$ be minimum such that $a_{m-n}\geq2$.Then, if $n$ is even or $n\geq6$, then $\beta_{m}\geq[0;1,1,1,1,1,1]>0.615$, and thus, since $\frac{\alpha_{m}+\beta_{m}}{(\alpha_{m}+2)(2-\beta_{m})}$ is increasing in $\alpha_{m}$ and in $\beta_{m}$,
		$$\frac{\alpha_{m}+\beta_{m}}{(\alpha_{m}+2)(2-\beta_{m})}>\frac{21.5+0.615}{23.5\times(2-0.615)}>0.679>\frac{\sqrt{493}}{33}$$
		a contradiction for $m$ large. Therefore $n=3$ or $n=5$, since 
		$$\beta_{m}=\beta>\frac{(1-\epsilon)\alpha+2}{2\alpha-(5-\epsilon)}>\frac{1-\epsilon}{2}>\frac{1}{3}, \hspace{0.5cm}\left(\epsilon<\frac{1}{3}\right)$$
		so $\beta_{m}=[0;a_{m-1},...]>\frac{1}{3}$, and thus $a_{m-1}\leq2$, and therefore $a_{m-1}=1$ (since, for $m$ large, $a_{m-1}\geq2$ implies $a_{m-1}\geq21$), if $n$ is always equal to $3$ then we are done.
		
		We will show that if $n\geq5$ for $m$ large then $a_{m}=21$,$a_{m+5j}=21,\forall j\in \mathbb{N}$ and $a_{n}=1$ if $n\geq m$ and $n\not\equiv m \pmod 5$. It's enough to prove that $a_{m+1}=a_{m+2}=a_{m+3}=a_{m+4}=1$, and $a_{m}=a_{m+5}=21$, and then use induction. First of all, if $a_{m}>21$ then $\alpha=\alpha_{m}>[22,1,1]=22.5$ and $\beta_{m}>[0;1,1,1,1]=0.6$, so
		$$\frac{\alpha_{m}+\beta_{m}}{(\alpha_{m}+2)(2-\beta_{m})}>\frac{22.5+0.6}{24.5\times1.4}>0.6734>\frac{\sqrt{493}}{33}$$
		contradiction for $m$ large.
		
		Therefore $a_{m}=21$. If we don't have $a_{m+1}=a_{m+2}=a_{m+3}=a_{m+4}=1$, then $a_{m+3}\geq2$. In this case, we should have $\alpha=\alpha_{m}<[21;1,1,21]<21.512<21.565$, a contradiction. This implies one claim.
		
		Now suppose that the $k_2^*(\alpha)$ belongs to $[\frac{2}{3},\frac{\sqrt{493}}{33})$. Then the structure of the continued factor of $\alpha$ is $a_{m},1,1,a_{m+3},1,1,a_{m+6},...,$ with $a_{m+3j}\geq21,\forall j\geq0$,for $m$ large enough. Let $$k=\liminf_{j\rightarrow \infty}a_{m+3j}.$$ 
		
		If $k=+\infty$ then $k_2^*(\alpha)$ is equal to $\frac{2}{3}$. Otherwise $k\geq21$, $k\in \mathbb{N}$, and for $n=m+6j$ large with $a_{n}=k$, we have
		\begin{align*}
			\frac{\alpha_{n}+\beta_{n}}{(\alpha_{n}-1)(1-\beta_{n})}&=\frac{1}{\alpha_{n}-1}+\frac{1}{1+\beta_{n}}\\&=[0;k-1,1,1,a_{n+3},1,1,...]+[0;1,1,1,a_{n-3},1,1,...]\\&>[0;k-1,1,1,k,2]+[0;1,1,1,k,2]
		\end{align*}
		If $k\leq144$, this is $\geq [0;143,1,1,k,2]+[0;1,1,1,144,2]>0.67286>\frac{\sqrt{493}}{33}$, a contradiction. So,indeed the structure of the continued fraction of $\alpha$ is $a_{m},1,1,a_{m+3},1,1,a_{m+6},...,$ with $a_{m+3j}\geq145,\forall j\geq0$, for all $m$ large enough with $a_{m}\geq2$, provided that $k_2^*(\alpha)\in [\frac{2}{3},\frac{\sqrt{493}}{33})$.
		
		This lower estimate of $145$ cannot be improved, since for $\alpha=[0;\overline{1,1,145}]$, since  $k_2^*(\alpha)=\frac{\sqrt{21317}}{217}=0.67282684<\frac{\sqrt{493}}{33}$.
		
	\end{proof}

	\subsection{Estimates on the Hausdorff dimension of $X(k)$}
	
	We now prove that $HD(X(k))>\frac{1}{2}$ for every sufficiently large $k$.
	We will use the dimension estimates for regular Cantor sets of Palis-Takens \cite[Page 68]{PT}.
	Given a Markov partition $\mathcal{R}^{1}=\{K_{1},...,K_{m}\}$ for a regular Cantor set and for $n\geq2$, let $\mathcal{R}^{n}$ denote the set of connected components of $\Psi^{-(n-1)}(K_{i}),K_{i}\in\mathcal{R^{1}}$ where $\Psi$ is the associated map. For $R\in\mathcal{R}^{n}$ take $\Lambda_{n,R}=\sup|(\Psi^{n})'|_{R}|$. Define $\gamma_{n}>0$ by 
	$$\sum_{R\in\mathbf{R}^{n}}(\Lambda_{n,R})^{-\gamma_{n}}=1$$
	We have that $HD(X(k))>\gamma_{n} \forall n\geq1$. So we can estimate $HD(X(k))$ by $\gamma_{n}$.
	
	Notice $X(k)$ is a regular Cantor set and the $\Psi$ for it is $G^{3}$,the third iterate of Gauss map,take the Markov partition as $\{R_{j}=\{[0;1,1,j,...]\}|k^{2}\leq j\leq k^{3}\}$.
	$$\Psi|_{R_{j}}(x)=\frac{q_{3}^{j}x-p_{3}^{j}}{-q_{2}^{j}+p_{2}^{j}}$$
	so
	$$\Psi'|_{R_{j}}(x)=\frac{(-1)^{n}}{(-q_{2}^{j}x+p_{2}^{j})^{2}}$$
	where $\dfrac{p^{j}_k}{q^{j}_k} = [0;b_{1}^{j},...,b_{k}^{j}]$ and $(1,1,j) = (b^{j}_{1}, b^{j}_{2}, b^{j}_{3})$,so $q_{3}^{j}=jq_{2}^{j}+q_{1}^{j}=2j+1$.
	
	To use the estimate from Palis-Takens, we use a classical estimate on the derivative of iterates of the Gauss map (cite Geometric properties of Lagrange and Markov spectra), which in our context translates as
	$(q_{3}^{j})^{2}\leq|\Psi'|_{R_{j}}|\leq4(q_{3}^{j})^{2}$. Then, as each branch $\Psi|_{R_i}$ is onto, we have that $HD(X(k))\geq \gamma_1=\gamma_1(k)$, where $\gamma_1$ is defined by
	$$\sum_{k^{2}\leq j\leq k^{3}}\sup|\Psi'|_{R_{j}}|^{-\gamma_{1}}=1$$
	This means that
	$$\sum_{k^2\leq j\leq k^3}(4q_3^j)^{-2\gamma_1}=16^{-\gamma_1}\sum_{k^{2}\leq j\leq k^{3}}(2j+1)^{-2\gamma_{1}}\leq1,\forall k$$
	and
	$$\sum_{k^2\leq j\leq k^3}(2j+1)^{-2\gamma_1}\leq 16^{\gamma_1}$$.
	
	But it is well known that
	$$\sum_{k^2\leq j\leq k^3}(2j+1)^{-1}$$
	is asymptotic to 
	$\frac{1}{2}(\log(k^3)-\log(k^2))=\frac{1}{2}\log(k)$. On the other hand, we have $\gamma_1(k)\leq HD(X(k))\leq 1$, which implies that $$\sum_{k^2\leq j\leq k^3}(2j+1)^{-2\gamma_1}\leq 16^{\gamma_1}$$ is uniformly bounded on $k$, and therefore $\gamma_1$ must be larger than $\frac{1}{2}$ if $k$ is large enough.

	\bibliographystyle{plain}
	\bibliography{bibliography}
	
\end{document}